\documentclass{amsart}
\usepackage{amsfonts,amsmath,color,graphicx,calligra,mathrsfs,tikz,tikz-cd,mathtools}
\usepackage{amsthm,amssymb}

\usetikzlibrary{arrows}

\usepackage[english]{babel}
\usepackage{comment}
\usepackage{hyperref}

\newtheorem{thm}{Theorem}[section]
\newtheorem{cor}[thm]{Corollary}
\newtheorem{lem}[thm]{Lemma}
\newtheorem{prop}[thm]{Proposition}
\theoremstyle{definition}
\newtheorem{defn}[thm]{Definition}
\theoremstyle{remark}
\newtheorem{rem}[thm]{Remark}
\theoremstyle{definition}
\newtheorem{ex}[thm]{Example}
\theoremstyle{definition}
\newtheorem{exs}[thm]{Examples}
\theoremstyle{remark}
\newtheorem{ques}[thm]{Question}
\newtheorem{sublemma}{Lemma}[thm]

\numberwithin{equation}{section}
\newcommand{\aaa}{\mathfrak{a}}
\newcommand{\bbb}{\mathfrak{b}}
\newcommand{\ccc}{\mathfrak{c}}
\newcommand{\ppp}{\mathfrak{p}}
\newcommand{\mmm}{\mathfrak{m}}
\newcommand{\nnn}{\mathfrak{n}}
\newcommand{\qqq}{\mathfrak{q}}
\newcommand{\AAA}{\mathfrak{A}}
\newcommand{\BBB}{\mathfrak{B}}

\newcommand{\ringa}{(A,\mathfrak{a},\mathfrak{A})}
\newcommand{\ringb}{(B,\mathfrak{b},\mathfrak{B})}

\newcommand{\locala}{(A,\mathfrak{m},k)}
\newcommand{\localb}{(B,\mathfrak{n},l)}
\newcommand{\colimi}{\underset{i}{\varinjlim} \;}
\newcommand{\limn}{\underset{n}{\varprojlim} \;}
\newcommand{\colimn}{\underset{n}{\varinjlim} \;}
\newcommand{\limi}{\underset{i}{\varprojlim} \;}
\newcommand{\Tor}{\mathrm{Tor}}
\newcommand{\kerr}{\mathrm{ker}}
\newcommand{\HHom}{\mathrm{Hom}}
\newcommand{\HH}{\mathrm{H}}

\makeatletter
\def\@tocline#1#2#3#4#5#6#7{\relax
  \ifnum #1>\c@tocdepth 
  \else
    \par \addpenalty\@secpenalty\addvspace{#2}%
    \begingroup \hyphenpenalty\@M
    \@ifempty{#4}{%
      \@tempdima\csname r@tocindent\number#1\endcsname\relax
    }{%
      \@tempdima#4\relax
    }%
    \parindent\z@ \leftskip#3\relax \advance\leftskip\@tempdima\relax
    \rightskip\@pnumwidth plus4em \parfillskip-\@pnumwidth
    #5\leavevmode\hskip-\@tempdima
      \ifcase #1
       \or\or \hskip 1em \or \hskip 2em \else \hskip 3em \fi%
      #6\nobreak\relax
    \hfill\hbox to\@pnumwidth{\@tocpagenum{#7}}\par
    \nobreak
    \endgroup
  \fi}
\makeatother

\begin{document}

\title[Formally regular rings and descent]{Formally regular rings and descent of regularity}%
\author{Samuel Alvite, Nerea G. Barral and Javier Majadas}%
\address{Departamento de Matem\'aticas, Facultad de Matem\'aticas, Universidad de Santiago de Compostela, E15782 Santiago de Compostela, Spain}%
\email{samuelalvite@gmail.com, nereabarral@gmail.com, j.majadas@usc.es}%

\thanks{$^{(\star)}$ This work was partially supported by Agencia Estatal de Investigaci\'on (Spain), grant MTM2016-79661-P (European FEDER support included, UE) and by Xunta de Galicia through the Competitive Reference Groups (GRC) ED431C 2019/10}


\begin{abstract}
  Valuation rings and perfectoid rings are examples of (usually non-noetherian) rings that behave in some sense like regular rings. We give and study an extension of the concept of regular local rings to non-noetherian rings so that it includes valuation and perfectoid rings and it is related to Grothendieck's definition of formal smoothness as in the noetherian case. For that, we have to take into account the topologies. We prove a descent theorem for regularity along flat homomorphisms (in fact for homomorphisms of finite flat dimension), extending some known results from the noetherian to the non-noetherian case, as well as generalizing some recent results in the non-noetherian case, such as the descent of regularity from perfectoid rings by B. Bhatt, S. Iyengar and L. Ma.
\end{abstract}

\maketitle
\setcounter{secnumdepth}{1}
\setcounter{section}{-1}
\tableofcontents

\section{Introduction}

Valuation rings and perfectoid rings, though they are in general non-noetherian, behave in some sense like regular rings. For instance, B. Bhatt, S. Iyengar and L. Ma prove in \cite{BIM} the following perfectoid version of a well known theorem by E. Kunz \cite{Kunz}  in positive characteristic: if $A\to B$ is a flat homomorphism (or more generally of finite flat dimension) where $(A,\mmm)$ is a noetherian local ring and $B$ perfectoid with $\mmm B\neq B$, then $A$ is regular. In fact, in \cite{BIM} more is proved, and their results give for instance the descent of regularity along morphisms that are coverings for the fpqc topology \cite[Theorem 10.4]{AF}. In this paper, we introduce a notion of regularity for not necessarily noetherian rings that we call formal regularity, and prove some results. Since among formally regular rings we find perfectoid rings, the above results are particular cases of our descent theorem:\\

\noindent \textbf{Theorem \ref{main3}.} \textit{Let $A\to B$ be a flat local homomorphism (or, more generally, of finite flat dimension) of (not necessarily noetherian) local rings. If $B$ is formally regular, then so is $A$.} \\

Bertin's definition of regularity in the non-noetherian case \cite{Ber} is chosen so that regularity localizes, and therefore it is formulated in terms of projective dimensions. However, we do not think that the localization of regularity in the non-noetherian context should be a natural requirement. Instead, we look for a concept agreeing with the usual definition in the noetherian case, including some rings expected to be ``regular" (polynomial rings over a field, valuation rings, perfectoid rings, ...), and that is related to Grothendieck's notion of formal smoothness in the same way as in the noetherian case. Much like topologies play an important role for formal smoothness when we do not have finiteness hypotheses, our definition of formal regularity will depend on the chosen topology.

Our main tool will be the cotangent complex of M. Andr\'e and D. Quillen which, contrary to its known rigidity in the noetherian case, is sensitive to the topologies when the rings are not noetherian. In fact, our definition of formal regularity appears implicitly or explicitly (depending on the chosen topology) in the work of M. Andr\'e. We give a fairly general definition of formal regularity of an ideal $\mmm$, though we mainly use two cases: formal regularity with the $\mmm$-adic topology or the stronger formal regularity with the discrete topology. Each one has its own advantages and therefore it is useful to work with both. In the noetherian case, formal regularity does not depend on the topology and, as expected, agrees with the notion of ideal locally generated by a regular sequence.

In the first section we define and study formal regularity of rings. In order to get some needed flexibility, we have to begin with a refinement of Grothendieck's definition of formal smoothness, so that it depends on an ideal and a topology (recovering Grothendieck's notion when the topology is the adic topology defined by that ideal). Next we define formal regularity for ideals and local rings, studying with some results and examples the dependence on the topology. Separated local formally regular rings are domains (Corollary \ref{cor.domain}), and perfectoid and valuation rings are formally regular (Corollary \ref{valuation.and.perfectoid}). The relation between formal smoothness and formal regularity is the one expected (Theorem \ref{teorem.f.smooth.regular} and Corollary \ref{cor.1.f.smooth.regular}): for instance, over a perfect field, a local ring is formally regular if and only if it is a formally smooth algebra. We also have to give a notion of regular homomorphisms for not necessarily noetherian rings (\ref{43} - \ref{48}), since we will need it at the end of the paper in order to generalize the relative form of Kunz's result. The last part of this section is devoted to studying the relation between vanishing of Koszul homology and formal regularity.

In section 2, we will give the main technical result (Theorem \ref{main}) needed in the following section. Since our notion of regularity is not based in the concept of projective dimension (it is ``based'' on rings instead of modules), our methods are necessarity different from those on \cite{BIM} (but note that some of the results obtained in that paper are valid not only for $A$-algebras but also for $A$-modules).

After deducing Theorem \ref{main3} above, section 3 is devoted to some other applications of the descent technical results of section 2. For instance, we obtain complete intersection and Gorenstein analogues of the above cited result by Bhatt, Iyengar and Ma of descent of regularity from perfectoid rings (Propositions \ref{prop.cidim} and \ref{Gorens}), and we can avoid the noetherian hypothesis in the original Kunz's result:\\

\noindent \textbf{Corollary \ref{absolute.Kunz}} \textit{Let $B$ be a local ring containing a field of characteristic $p>0$ and $\phi:B\to B$ the Frobenius homomorphism. If $\phi$ is flat (or more generally of finite flat dimension) then $B$ is formally regular.} \\

In fact, we drop the noetherian hypothesis also in the more general relative version by André, Dumitrescu and Radu of Kunz's theorem (Theorem \ref{relativekunz}).

\section{Formal regularity}

All rings and algebras will be commutative (but note that in Section 2 sometimes they will be graded anti-commutative). We denote by $\mathbb{L}_{B|A}$ the cotangent complex of an $A$-algebra $B$, and if $M$ is a $B$-module, by $\HH^i(A,B,M)=\HH^i(\HHom_B(\mathbb{L}_{B|A},M))$, $\HH_i(A,B,M)=\HH_i(\mathbb{L}_{B|A}\otimes_B M)$ the Andr\'e-Quillen (co)homology modules \cite{An1974}, \cite{QMIT}. A local ring $\locala$ is a (not necessarily noetherian) ring $A$ with a unique maximal ideal $\mmm$ and residue field $k$.

\begin{defn}\label{topideal}
We will consider a ring $A$ with a decreasing filtration $\AAA=\{\aaa_n\}_{n>0}$ consisting on ideals of $A$. We denote these data by $(A,\AAA)$, and for the sake of simplicity we say that $(A,\AAA)$ is a \emph{topological ring}, since the results below are not affected by a change in the filtration inducing the same topology. Sometimes we fix an ideal $\aaa$ of $A$ such that $\aaa_n \subset \aaa$ for all $n$. We denote these data by $\ringa$. Given an ideal $\aaa$, most of the time we will work with two filtrations, the $\aaa$-\emph{adic} one defined by $\aaa_n=\aaa^n$ and the \emph{discrete} one (the 0-adic) defined by $\aaa_n=0$ for each $n$. However, non-adic topologies arise naturally in the non-noetherian case as the next example shows (see also Example \ref{exampleformalregular} (i)).

A homomorphism of topological rings (or a continuous homomorphism)
\[f:(A,\AAA) \to (B,\BBB)\;{\rm or }\;f:\ringa \to (B,\BBB) \; {\rm or } \; f:(A,\AAA) \to \ringb \]
is a ring homomorphism $f:A\to B$ such that for each $n$ there exists some $t$ such that $f(\aaa_t)\subset \bbb_n$, while if we write
\[ f:\ringa \to \ringb\]
we are also assuming that $f(\aaa)\subset \bbb$. We also say that $(B,\BBB)$ is a topological $(A,\AAA)$-algebra (or that $\ringb$ is a topological $\ringa$-algebra when $f(\aaa)\subset \bbb$). Sometimes, deleting some terms of the filtration $\AAA$ and renumbering we will assume that $f(\aaa_n)\subset \bbb_n$ for all $n$ (see Remark \ref{independent}).
\end{defn}

\begin{ex}\label{examplecompletion}
Let $\locala$ be a local ring. Consider the $\mmm$-adic filtration $\mathfrak{M}=\{\mmm^n\}_{n>0}$. Let
\[\hat{A}:= \limn A/\mmm^n \]
be its completion. The ring $\hat{A}$ is local \cite[III, §2, n.13, Proposition 19]{BoAC14} with maximal ideal
\[ \hat{\mmm} := \mathop{\lim_{\longleftarrow}}_{n} \mmm/\mmm^n \]
and residue field $k$. We have three natural topologies on $\hat{A}$: the discrete topology, the $\hat{\mmm}$-adic topology, and the limit topology $\{\widehat{\mmm^n}\}_{n>0}$ given by

\[ \widehat{\mmm^i} := \mathop{\lim_{\longleftarrow}}_{n>i} \mmm^i/\mmm^n = \kerr (\hat{A}\xrightarrow{\epsilon_i} A/\mmm^i) \]

The two latter agree if dim$_k(\mmm/\mmm^2)<\infty$ (for instance if $A$ is noetherian), since then $\hat{\mmm}^n= \widehat{\mmm^n}$ \cite[$0_{\rm I}$.7.2.7]{EGAI}, but in general these topologies are different: while the topological ring $(\hat{A}, \hat{\mmm}, \{\widehat{\mmm^n}\}_{n>0})$ is clearly separated and complete, $\hat{A}$ is not necessarily complete for the $\hat{\mmm}$-adic topology \cite[III, §2, Exercise 12]{BoAC14}.

\end{ex}

\hfill \break
\subsection{Formal smoothness}
\hfill \break

We will need a slight refinement of Grothendieck's definition of formal smoothness.

\begin{defn}\label{formallysmooth}
We say that a topological $(A,\AAA)$-algebra $\ringb$ is \emph{formally smooth} (or that $B$ is formally smooth over $(A,\AAA)$ for the ideal $\bbb$ with the $\BBB$ topology) if the following condition is satisfied: for any discrete topological $(A,\AAA)$-algebra $(C,\{(0)\}_{n>0})$, any square-zero ideal $M$ of $C$ (and then $M$ is a $C/M$-module), and any continuous $A$-algebra homomorphism $f:B\to C/M$ (that is, $f(\bbb_t)=0$ for some $t$) such that $f$ induces a $B/\bbb$-module structure on $M$ (that is, $p^{-1}(f(\bbb))M=0$, where $p:C\to C/M$ is the canonical map) there exists a continuous $A$-algebra homomorphism $g:B\to C$ such that $f=pg$, where $p:C \to C/M$ is the canonical map.
\end{defn}

\begin{rem}\label{independent}
Clearly, this definition depends only on the topologies induced by the filtrations $\AAA$ and $\BBB$ and not on the filtrations themselves.
\end{rem}

We will show (Corollary \ref{Grothendiecksmooth}) that when $\BBB$ is the $\bbb$-adic filtration (and $\AAA$ is an adic filtration), this definition does not depend on $\AAA$ and is equivalent to Grothendieck's definition  \cite[$0_{\rm IV}$.19.3.1, $0_{\rm IV}$.19.4.3]{EGAIV1} of $B$ being formally smooth over $A$ for the $\bbb$-adic topology. However, distinguishing from the ideal $\bbb$ and the topology $\BBB$ will be important (in the non-noetherian case). While we will see (Corollary \ref{noetheriansmooth}) that when $B$ is \emph{noetherian}, $\AAA$ is an adic filtration and $\BBB$ is the $\widetilde{\bbb}$-adic filtration for some ideal $\widetilde{\bbb}\subset\bbb$, Definition \ref{formallysmooth} does not depend on the ideal $\widetilde{\bbb}$ (and therefore, for any ideal $\widetilde{\bbb}$ it is equivalent to Grothendieck's definition), in the non-noetherian case in general it depends on $\widetilde{\bbb}$ (Example \ref{exampleformal}) and in particular considering on the ideal $\bbb$ the $\bbb$-adic topology or the discrete one will be different and important.\\

We will need the following proposition similar to \cite[$0_{\rm IV}$.19.3.10]{EGAIV1}
\begin{prop}\label{btildeb}
Let $(B,\bbb, \BBB)$ be a topological $(A,\AAA)$-algebra. The following are equivalent:\\

(i) $(B,\bbb, \BBB)$ is formally smooth over $(A,\AAA)$.\\

(ii) For any topological $(A,\AAA)$-algebra $(C,\{M^n\}_{n>0})$, where $M$ is an ideal of $C$, and any continuous $A$-algebra homomorphism $f:B\to C/M$ (that is, $f(\bbb_t)=0$ for some $t$) such that $p^{-1}(f(\bbb))M\subset M^2$ there exists an $A$-algebra homomorphism $g:B\to \hat{C}=\limi C/M^i$ continuous for the limit topology in $\hat{C}$ such that $f=\hat{p}g$, where $p:C\to C/M$ is the canonical map and $\hat{p}:\hat{C} \to C/M$ its completion.
\end{prop}

\begin{proof}
The condition in Definition \ref{formallysmooth} is a particular case of (ii) when $M^2=0$, and for (i)$\Rightarrow$(ii) we define $g={\varprojlim }\; g_i:B\to \varprojlim C/M^i$, where the maps $g_i$ are constructed inductively with the diagrams
\[
\begin{tikzcd}
& B \arrow[ld, swap, "g_{i+1}", dashrightarrow] \arrow[d, "g_i"]\\
C/M^{i+1}\arrow[r] & C/M^i
\end{tikzcd}
\]
since the condition $f(\bbb)M\subset M^2$ gives us that for all $i>0$, $f$ (and so inductively $g_i$) induces a $B/\bbb$-module structure on $M^i/M^{i+1}$.
\end{proof}

\begin{defn} \cite[18.4.2]{EGAIV1} Let $A$ be a ring, $B$ an $A$-algebra. If $p:E\to B$ is a surjective homomorphism of $A$-algebras whose kernel $M$ is a square-zero ideal of $E$, we will say that the exact sequence
\[ 0\to M \to E\to B \to 0 \]
is an \emph{extension} of $B$ over $A$ by $M$. Note that $M$ has then a canonical $B$-module structure. We will say that two extensions of $B$ over $A$ by $M$
\[ 0\to M \to E\to B \to 0, \quad \quad  0\to M \to E'\to B \to 0 \]
are equivalent if there exists an $A$-algebra homomorphism $E\to E'$ making commutative the diagram
\[
\begin{tikzcd}
0 \arrow[r]& M \arrow[r]\arrow[d, equal]&E \arrow[r]\arrow[d]&B \arrow[r]\arrow[d, equal]&0\\
0 \arrow[r]& M \arrow[r]&E' \arrow[r]&B \arrow[r]&0
\end{tikzcd}
\]
Then $E\to E'$ is an isomorphism and this is an equivalence relation. We will say that an extension
\[ 0\to M \to E\xrightarrow{p} B \to 0 \]
is trivial if $p$ has an $A$-algebra section. All trivial extensions form one equivalence class.

If
\[ 0\to M \to E\to B \to 0 \]
is an extension and $B'\to B$ an $A$-algebra homomorphism, then there is an associated extension
\[ 0\to M \to E\times_BB'\to B' \to 0. \]
Similarly, if $M\to M'$ is a $B$-module homomorphism, we have an extension
\[ 0\to M' \to E\oplus_MM'\to B \to 0. \]
Finally, if $A'\to A$ is a ring homomorphism, we have an associated extension over $A'$
\[ 0\to M \to E\to B \to 0 \]
\cite[16.8, 16.9, 16.7]{An1974}.
\end{defn}

\begin{prop}\label{extensionH1} \cite[16.12]{An1974}
Let $A$ be a ring, $B$ an $A$-algebra and $M$ a $B$-module. There exists a bijection, natural on $A$, $B$ and $M$ in the above sense, between the set of equivalence classes of extensions of $B$ over $A$ by $M$ and the cohomology module $\HH^1(A,B,M)$, taking the class of the trivial extensions to 0.
\end{prop}

\begin{prop}\label{formallysmoothideal}
Let $f:(A,\AAA) \to \ringb$ be a continuous homomorphism, and assume for simplicity (see Remark \ref{independent}) that $f(\aaa_n) \subset \bbb_n$ for all $n$. The following are equivalent:\\
(i) The $(A,\AAA)$-algebra $\ringb$ is formally smooth.\\
(ii) $\colimn \HH ^1(A/\aaa_n,B/\bbb_n,M)=0$ for any $B/\bbb$-module $M$.
\end{prop}
\begin{proof}
(ii)$\Rightarrow$(i) Let $C$ be an $A$-algebra such that for some $r$ the image of $\aaa_r$ in $C$ is zero, $M$ a square-zero ideal of $C$, $f:B\to C/M$ an $A$-algebra homomorphism such that $f(\bbb_t)=0$ for some $t$ and that $p^{-1}(f(\bbb))M=0$ where $p:C\to C/M$ is the canonical map. Then $f$ induces a homomorphism of $A$-algebras $B/\bbb_t \to C/M$. Consider the cartesian square defining $C_t$
\[
\begin{tikzcd}
C_t \arrow[r] \arrow[d] & B/\bbb_t \arrow[d] \\\
C \arrow[r] & C/M
\end{tikzcd}
\]

By Proposition \ref{extensionH1}, enlarging $t$ if necessary so that $t\geq r$, the extension
\[ 0\to M \to C_t \to B/\bbb_t \to 0 \]
corresponds to an element of $\HH^1(A/\aaa_t,B/\bbb_t,M)$. By (ii), this element goes to zero in $\HH^1(A/\aaa_s,B/\bbb_s,M)$ for some $s\geq t$. Since the bijection of Proposition \ref{extensionH1} takes the trivial extension to zero, the extension
\[ 0\to M \to C_s \to B/\bbb_s \to 0 \]
is trivial, that is, there exists a section of $A$-algebras $\sigma :B/\bbb_s \to C_s$. The required map $g$ is the composition $B\to B/\bbb_s \to C_s \to C$.

(i)$\Rightarrow$(ii) Let $M$ be a $B/\bbb$-module and $x$ an element of $\HH^1(A/\aaa_n,B/\bbb_n,M)$. Let
\[ 0\to M \to C \to B/\bbb_n \to 0 \]
be an extension over $A/\aaa_n$ associated to $x$, where $M^2=0$ as an ideal of $C$.
By (i), there exists an $A$-algebra homomorphism $g:B\to C$, continuous for the discrete topology on $C$, making commutative the triangle
\[
\begin{tikzcd}
& B \arrow[ld, swap, "g"] \arrow[d]\\
C \arrow[r] & B/\bbb_n
\end{tikzcd}
\]
Since $g$ is continuous, we can take $t\geq n$ such that $g(\bbb_t)=0$, and then $g$ induces a map $h:B/\bbb_t \to C$. The image $y$ of $x$ in $\HH^1(A/\aaa_t,B/\aaa_t,M)$ corresponds to the extension
\[ 0\to M \to C\times_{B/\bbb_n}B/\bbb_t \to B/\bbb_t \to 0 \]
and the map $h$ induces a section of this extension. Therefore $y=0$.
\end{proof}

\begin{cor}\label{Grothendiecksmooth}
Let $f: A\to B$ be a ring homomorphism, $\aaa$ an ideal of $A$, $\bbb$ an ideal of $B$ such that $f(\aaa)\subset \bbb$. Let $\AAA$ be the $\aaa$-adic topology on $A$ and $\BBB$ the $\bbb$-adic one on $B$. The following are equivalent:\\
(i) $(A,\AAA) \to (B,\bbb,\BBB)$ is formally smooth.\\
(i') $(A,\mathfrak{O}) \to (B,\bbb,\BBB)$ is formally smooth where $\mathfrak{O}$ is the 0-adic topology. \\
(ii) $B$ is a formally smooth $A$-algebra for the $\bbb$-adic topology in the sense of Grothendieck \cite[$0_{\rm IV}$.19.3.1]{EGAIV1}.
\end{cor}
\begin{proof}

By \ref{formallysmoothideal}, (i) is equivalent to
\[\colimn \HH^1(A/\aaa^n,B/\bbb^n,M) =0 \]
for any $B/\bbb$-module $M$, and (i') is equivalent to
\[\tag{*} \colimn \HH^1(A,B/\bbb^n,M) =0 \]
for any $B/\bbb$-module $M$. By \cite[2.2.5]{MR}, (ii) is also equivalent to (*).

Then, the result follows from the Jacobi-Zariski exact sequence \cite[5.1]{An1974}
\[ \colimn \HH^0(A,A/\aaa^n,M) \to \colimn \HH^1(A/\aaa^n,B/\bbb^n,M) \to \]
\[\colimn \HH^1(A,B/\bbb^n,M) \to \colimn \HH^1(A,A/\aaa^n,M) \]
since $\colimn \HH^0(A,A/\aaa^n,M)=0$ \cite[6.3]{An1974} and $\colimn \HH^1(A,A/\aaa^n,M)=0$ \cite[2.3.4]{MR}
\end{proof}

Therefore, from now on, if no topology is specified on $A$, ``formally smooth over $A$'' must be understood with the discrete topology on $A$.

\begin{cor}\label{noetheriansmooth}
Let $f: A\to B$ be a ring homomorphism, $\aaa$ and ideal of $A$, $\widetilde{\bbb}\subset \bbb$ ideals of $B$ such that $f(\aaa)\subset \widetilde{\bbb}$, and let $\AAA$, $\widetilde{\BBB}$, $\BBB$ the adic topologies induced by these ideals. If $B$ is noetherian the following are equivalent:\\
(i) $(A,\AAA) \to (B,\bbb,\widetilde{\BBB})$ is formally smooth.\\
(ii) $(A,\AAA) \to (B,\bbb,\BBB)$ is formally smooth.
\end{cor}
\begin{proof}
By Corollary \ref{Grothendiecksmooth}, we can assume $\aaa=0$. Let $M$ be a $B/\bbb$-module. Since $B$ is noetherian we have \cite[2.3.4]{MR}
\[ \colimn \HH^1(B,B/\widetilde{\bbb}^n,M)=0=\colimn \HH^2(B,B/\widetilde{\bbb}^n,M). \]
Then the Jacobi-Zariski exact sequence \cite[5.1]{An1974}
\[ \HH^1(B,B/\widetilde{\bbb}^n,M) \to \HH^1(A,B/\widetilde{\bbb}^n,M) \to \HH^1(A,B,M) \to \HH^2(B,B/\widetilde{\bbb}^n,M) \]
gives an isomorphism
\[ \colimn \HH^1(A,B/\widetilde{\bbb}^n,M) = \HH^1(A,B,M). \]
Since the term on the right does not depend on $\widetilde{\bbb}$, Proposition \ref{formallysmoothideal} gives the result.
\end{proof}

\begin{ex}\label{exampleformal}
Corollary \ref{noetheriansmooth} is false without the noetherian hypothesis. In \cite[2.3.6, 2.3.7]{MR} there is an example of an $A$-algebra $B$ formally smooth for the $J$-adic topology in the sense of Grothendieck, where $J$ is an ideal of $B$, and a $B/J$-module $M$ such that $\HH^1(A,B,M)\neq 0$. By Corollary \ref{Grothendiecksmooth} this algebra is formally smooth for the ideal $J$ with the $J$-adic topology, and by Proposition \ref{formallysmoothideal} it is not formally smooth for the ideal $J$ with the discrete ($0$-adic) topology.

However, if $A$ is a ring, $B$ an $A$-algebra, $\bbb_1\subset \bbb_2\subset \bbb$ ideals of $B$ and $B$ is a formally smooth $A$-algebra for the ideal $\bbb$ (resp. $\bbb_1$) with the $\bbb_1$-adic topology then it is also formally smooth for the ideal $\bbb$ (resp. $\bbb$ or $\bbb_2$) with the $\bbb_2$-adic topology. This fact follows easily from the definition or from the Jacobi-Zariski exact sequence
\[ \HH^1(B/\bbb_1^n,B/\bbb_2^n,M) \to \HH^1(A,B/\bbb_2^n,M) \to \HH^1(A,B/\bbb_1^n,M)  \]
since
\[ \colimn \HH^1(B/\bbb_1^n,B/\bbb_2^n,M)=0 \]
as we will see in the proof of Proposition \ref{prop.formal.regular.topologies}.
\end{ex}

\begin{ex}\label{smooth.non.flat}
Grothendieck shows in \cite[$0_{\rm IV}$.19.7.1]{EGAIV1} that if $\locala \to \localb$ is a local homomorphism of noetherian local rings and $B$ is formally smooth over $A$ for the ideal $\nnn$ with the $\nnn$-adic topology, then $B$ is a flat $A$-module. This does not longer hold if $A$ is not noetherian: there exists a local homomorphism of rings $\locala \to \localb$ such that $B$ is noetherian and formally smooth for the ideal $\nnn$ with the $\nnn$-adic topology and $B$ is not flat over $A$. It suffices to take a rank 1 non-discrete valuation ring $\locala$, and $B:=k$. We have $\HH^1(A,k,M)=\HHom_k(\mmm/\mmm^2,M)$ by \cite[6.1]{An1974}, and this last module vanishes since  $\mmm=\mmm^2$, showing that $k$ is formally smooth over $A$ for its maximal ideal $(0)$. However, $A\to k$ is not flat (since a flat local homomorphism is injective).
\end{ex}

\hfill \break
\subsection{Formal regularity}
\hfill \break

\begin{defn}\label{defnformallyregular}
Let $\ringa$ as in Definition \ref{topideal}. We say that the ideal $\aaa$ is \emph{formally regular} with the $\AAA$ topology if
\[ \colimn \HH^2(A/\aaa_n,A/\aaa,M)=0 \]
for any $A/\aaa$-module $M$.

If $\locala$ is a local ring, $\aaa =\mmm$ its maximal ideal and $\AAA$ the $\mmm$-adic topology (resp. the $(0)$-adic topology), we will say that $A$ is formally regular with the $\mmm$-adic topology (resp. the discrete topology).
\end{defn}

\begin{rem}\label{remark.symmetrically}
When the topology $\AAA$ is the $\aaa$-adic topology, this is what M. Andr\'e calls symmetrically regular ideal in \cite{An1974a}. In this case, he proves that $\aaa$ is formally regular with the $\aaa$-adic topology if and only if $\aaa/\aaa^2$ is projective as $A/\aaa$-module and the canonical homomorphism
\[ {\rm S}_{A/\aaa}(\aaa/\aaa^2) \to {\rm Gr}_{\aaa}(A) \]
is an isomorphism, where S denotes the symmetric algebra, and ${\rm Gr}_{\aaa}(A)=\oplus_{n\geq 0}\aaa^n/\aaa^{n+1}$ is the associated graded algebra \cite[12.2]{An1974}.
\end{rem}

\begin{exs}\label{exampleformalregular}
(i) Let $\locala$ be a local ring, $(\hat{A}, \hat{\mmm}, k)$ its $\mmm$-adic completion. Then $A$ is formally regular with the $\mmm$-adic topology if and only if $\hat{A}$ is formally regular with the limit topology $\{\widehat{\mmm^n}\}_{n>0}$ (see Example \ref{examplecompletion}). This is clear since for all $n$ we have $A/\mmm^n = \hat{A}/\widehat{\mmm^n}$.\\
(ii) If $\aaa_n=\aaa$ for all $n$ then $\aaa$ is formally regular with the $\AAA$-adic topology. For instance, if $\aaa^2=\aaa$, then $\aaa$ is formally regular with the $\aaa$-adic topology.\\
\end{exs}

\begin{prop}\label{prop.formal.regular.topologies}
Let $A$ be a ring, $\aaa_2 \subset \aaa_1 \subset \aaa$ ideals of $A$.\\
(i) If $\aaa$ is formally regular with the $\aaa_2$-adic topology, then it is formally regular with the $\aaa_1$-adic topology. In particular, if an ideal $\aaa$ is formally regular with the discrete topology it is formally regular with the $\aaa$-adic topology.\\
(ii) If $A$ is noetherian, then the converse of (i) is true.
\end{prop}
\begin{proof}
For $j=1,2$ we have
\[\colimn \HH^i(A,A/\aaa_j^n,M)=0 \]
for $i=0,1$ and any $A/\aaa$-module $M$, and when $A$ is noetherian also for $i=2$ \cite[2.3.4]{MR}. Therefore from the Jacobi-Zariski exact sequence
\[ \colimn \HH^i(A,A/\aaa_2^n,M)\to \colimn \HH^{i+1}(A/\aaa_2^n,A/\aaa_1^n,M)\to \colimn \HH^{i+1}(A,A/\aaa_1^n,M) \]
we obtain
\[ \colimn \HH^1(A/\aaa_2^n,A/\aaa_1^n,M) =0 \]
and, if $A$ is noetherian
\[ \colimn \HH^2(A/\aaa_2^n,A/\aaa_1^n,M) =0. \]
Now the result follows from the Jacobi-Zariski exact sequence
\[ \colimn \HH^1(A/\aaa_2^n,A/\aaa_1^n,M) \to \colimn \HH^2(A/\aaa_1^n,A/\aaa,M) \]
\[\to \colimn \HH^2(A/\aaa_2^n,A/\aaa,M) \to \colimn \HH^2(A/\aaa_2^n,A/\aaa_1^n,M). \]
\end{proof}

\begin{ex}\label{ex.adic.discrete}
Proposition \ref{prop.formal.regular.topologies} (ii) is not true without the noetherian hypothesis. We are going to see that there exists a local ring $(C,\nnn,l)$ formally regular with the $\nnn$-adic topology but not with the discrete topology, that is,
\[ \HH^2(C,l,l)\neq 0 = \colimn \HH^2(C/\nnn^n,l,l). \]

By \cite[2.3.6, 2.3.7]{MR} there exists a ring $A$ with a maximal ideal $N$ and another ideal $I\subset N$ such that:\\
(i) $N=N^2$,\\
(ii) $B:=A/I$ is a formally smooth $A$-algebra for the ideal $J:=N/I$ with the $J$-adic topology, and\\
(iii) $\HH^1(A,B,l)\neq 0$, where $l=B/J$.

Condition (i) implies that $N$ is formally regular with the $N$-adic topology, and Corollary \ref{Grothendiecksmooth} and condition (ii) implies that $B:=A/I$ is a formally smooth $A$-algebra for the ideal $J:=N/I$ with the $J$-adic topology considering in $A$ the $N$-adic topology. Therefore, by Theorem \ref{teorem.f.smooth.regular} below, $J$ is formally regular with the $J$-adic topology.

Consider the Jacobi-Zariski exact sequence
\[ \HH^1(A,l,l) \to \HH^1(A,B,l) \to \HH^2(B,l,l). \]
The first term vanishes, since $\HH^1(A,l,l)=\HHom_l(N/N^2,l)$, and the second one is not zero by (iii). Therefore $\HH^2(B,l,l)\neq 0$ as desired.

Finally, take $C:=B_J$, $\nnn=JB_J$, and use that Andr\'e-Quillen cohomology localizes \cite[5.27]{An1974}.
\end{ex}

\begin{cor}\label{noetherian.regular}
Let $A$ be a noetherian ring, $\aaa$ and ideal of $A$. The following are equivalent:\\
(i) The ideal $\aaa$ is formally regular with the $\aaa$-adic topology.\\
(ii) The ideal $\aaa$ is formally regular with the discrete topology.\\
(iii) The ideal $\aaa$ is locally generated by a regular sequence.
\end{cor}
\begin{proof}
The equivalence of (i) and (ii) follows from Proposition \ref{prop.formal.regular.topologies}, and the equivalence of (ii) and (iii) from \cite[6.25]{An1974} having in mind that Andr\'e-Quillen cohomology localizes under our hypothesis \cite[4.59, 5.27]{An1974}.
\end{proof}

\begin{prop}\label{prop.regular.projective}
Let $\ringa$ as in Definition \ref{topideal}, and assume that there exists some $t>0$ such that $\aaa_t \subset \aaa^2$ (for instance if $\AAA$ is the $\aaa$-adic or the discrete topology). If $\aaa$ is formally regular with the $\AAA$ topology, then the $A/\aaa$-module $\aaa/\aaa^2$ is projective.
\end{prop}
\begin{proof}
We have to show that the functor
\[ \HHom_{A/\aaa}(\aaa/\aaa^2,-) \]
is exact. Let
\[ 0\to M' \to M \to M''\to 0 \]
be an exact sequence of $A/\aaa$-modules. This induces an exact sequence
\[ \HH^0(A/\aaa_n,A/\aaa,M'')\to  \HH^1(A/\aaa_n,A/\aaa,M') \to  \]
\[ \HH^1(A/\aaa_n,A/\aaa,M) \to \HH^1(A/\aaa_n,A/\aaa,M'') \to \HH^2(A/\aaa_n,A/\aaa,M') \]

The first module vanishes since $A/\aaa_n \to A/\aaa$ is surjective \cite[6.3]{An1974}. By hypothesis, $\colimn \HH^2(A/\aaa_n,A/\aaa,M') =0$, so we obtain an exact sequence
\[ 0\to \colimn \HH^1(A/\aaa_n,A/\aaa,M') \to  \colimn \HH^1(A/\aaa_n,A/\aaa,M) \to \colimn \HH^1(A/\aaa_n,A/\aaa,M'') \to 0 \]
By \cite[6.1]{An1974}, this exact sequence can be written as
\[ 0\to \colimn \HHom_{A/\aaa}(\aaa/(\aaa^2+\aaa_n),M') \to  \]
\[\colimn \HHom_{A/\aaa}(\aaa/(\aaa^2+\aaa_n),M) \to \colimn \HHom_{A/\aaa}(\aaa/(\aaa^2+\aaa_n),M'') \to 0 \]
and since $\aaa_t \subset \aaa^2$, we obtain an exact sequence
\[ 0\to  \HHom_{A/\aaa}(\aaa/\aaa^2,M') \to  \HHom_{A/\aaa}(\aaa/\aaa^2,M) \to  \HHom_{A/\aaa}(\aaa/\aaa^2,M'') \to 0 \]
\end{proof}

\begin{rem}\label{provanishing}
Let $A$ be a ring, $\aaa$ an ideal of $A$ formally regular with the $\aaa$-adic topology. Since inductive limits are exact, for any $n\geq 2$, any $t\geq n$ and any $A/\aaa$-module $M$, taking inductive limits in $t$ we see as in  \cite[12.7]{An1974} that the connecting homomorphism in the Jacobi-Zariski exact sequence
\[ \HHom_{A/\aaa}(\aaa^n/\aaa^{n+1},M) = \HH^1(A/\aaa^{n+t},A/\aaa^n,M) \to \HH^2(A/\aaa^n,A/\aaa,M)  \]
is an isomorphism, and then for all $n\geq 2$ the homomorphisms
\[\HH^2(A/\aaa^n,A/\aaa,M) \to \HH^2(A/\aaa^{n+1},A/\aaa,M) \]
are zero. In order to compare with the case of homology, we note the following weaker statement: if $\colimn \HH^2(A/\aaa^n,A/\aaa,M)=0$ then for any $r$ there exists some $s>r$ such that the homomorphism
\[\HH^2(A/\aaa^r,A/\aaa,M) \to \HH^2(A/\aaa^s,A/\aaa,M) \]
vanishes.

For homology (instead cohomology) M. Andr\'e \cite[Proposition A]{An1974a} proved that if $A$ is a ring and $\aaa$ an ideal of $A$, the following are equivalent:\\
(i) For any $A/\aaa$-module $M$ and any $r$ there exist some $s>r$ such that the homomorphism
\[\HH_2(A/\aaa^s,A/\aaa,M) \to \HH_2(A/\aaa^r,A/\aaa,M) \]
vanishes, that is
\[\{ \HH_2(A/\aaa^n,A/\aaa,M)\}_{n>0}=0 \]
as pro-module (in the sense of \cite[Appendix]{AM}).\\
(ii) $\aaa/\aaa^2$ is flat as $A/\aaa$-module and the canonical homomorphism
\[ {\rm S}_{A/\aaa}(\aaa/\aaa^2) \to {\rm Gr}_{\aaa}(A) \]
is an isomorphism (compare with Remark \ref{remark.symmetrically}).

In particular, we obtain that an ideal $\aaa$ of a ring $A$ is formally regular with the $\aaa$-adic topology if and only if $\aaa/\aaa^2$ is projective as $A/\aaa$-module and
\[ \{ \HH_2(A/\aaa^n,A/\aaa,M)\}_{n>0}=0. \]
With the discrete topology we have a similar result:
\end{rem}

\begin{prop}
Let $\aaa$ be an ideal of a ring $A$. The following are equivalent:\\
(i) $\aaa$ is formally regular with the discrete topology.\\
(ii) $\HH_2(A,A/\aaa,M)=0$ for any $A/\aaa$-module $M$ and $\aaa/\aaa^2$ is projective over $A/\aaa$.
\end{prop}
\begin{proof}
(ii)$\Rightarrow$(i) The functor $\HH^1(A,A/\aaa,-)=\HHom_{A/\aaa}(\aaa/\aaa^2,-)$ is exact, since $\aaa/\aaa^2$ is projective, and then we have an isomorphism \cite[3.21]{An1974}
\[ \HH^2(A,A/\aaa,M)=\HHom_{A/\aaa}(\HH_2(A,A/\aaa,A/\aaa),M). \]
(i)$\Rightarrow$(ii) It follows from Proposition \ref{prop.regular.projective} and taking $C_*= \mathbb{L}_{A/\aaa|A}$ in the following elementary lemma, for which we have been unable to find a reference.
\end{proof}

\begin{lem}
Let $n$ be an integer, $A$ a ring and $C_*$ a complex of $A$-modules. If
\[ \HH^n(\HHom_A(C_*,M))=0 \]
for any $A$-module $M$, then
\[ \HH_n(C_*\otimes_A M)=0 \]
for any $A$-module $M$.
\end{lem}
\begin{proof}
If $D\to E \to F$ is a sequence of $A$-modules such that
\[\HHom_A(F,T) \to \HHom_A(E,T) \to \HHom_A(D,T) \]
is exact for all $T$, then $D\to E \to F$ is exact.
So in order to see that $C_*\otimes_A M$ is exact at $n$ it suffices to see that $\HHom_A(C_*\otimes_A M,T)=\HHom_A(C_*,\HHom_A(M,T))$ is exact at $n$ for any $T$, which holds by hypothesis.
\end{proof}

\begin{lem}\label{lema.domain}
Let $A$ be a ring, $\ppp$ a prime ideal of $A$, $\aaa = \cap_{n>0}\ppp^n$. If $\ppp$ is formally regular with the $\ppp$-adic topology, then $\aaa$ is a prime ideal.
\end{lem}
\begin{proof}
Replacing $A$, $\ppp$, $\aaa$ by $A/\aaa$, $\ppp/\aaa$, $(0)$, we can assume that $A$ is separated for the $\ppp$-adic topology and we have to prove that $A$ is a domain. By Proposition \ref{prop.regular.projective}, $\ppp/\ppp^2$ is a projective $A/\ppp$-module, and as we saw in Remark \ref{remark.symmetrically} we have an isomorphism
\[ {\rm S}_{A/\ppp}(\ppp/\ppp^2) \to {\rm Gr}_{\ppp}(A). \]
Since $\ppp/\ppp^2$ is a summand of a free $A/\ppp$-module, ${\rm S}_{A/\ppp}(\ppp/\ppp^2)$ is a retract of a polynomial $A/\ppp$-algebra, and in particular it is a domain. Therefore ${\rm Gr}_{\ppp}(A)$ is a domain. Let $x,y\in A$, $x\neq 0 \neq y$. Since $A$ is separated, there exist $r\geq 0$, $s\geq 0$ such that $x\in \ppp^r - \ppp^{r+1}$, $y\in \ppp^s - \ppp^{s+1}$, and then $\bar{x}\in \ppp^r / \ppp^{r+1}$, $\bar{y}\in \ppp^s / \ppp^{s+1}$ are non-zero elements of the domain ${\rm Gr}_{\ppp}(A)$. Therefore $\bar{x}\bar{y} \neq 0$ and so $xy \neq 0$.
\end{proof}

\begin{prop}\label{prop.domain}
Let $\locala$ be a local ring formally regular with the $\mmm$-adic topology. Then $\hat{A}$ is a domain.
\end{prop}
\begin{proof}
The local ring ($\hat{A},\hat{\mmm},k)$ is complete and separated for the topology defined by the filtration $\{\nnn_n\}_{n>0}$, where $\nnn_n = \widehat{\mmm^n}$. The graded ring ${\rm gr}(\hat{A}):= \hat{A}/\nnn_1 \oplus \nnn_1/\nnn_2 \oplus \dots$ is isomorphic to the graded ring ${\rm Gr}_{\mmm}(A):=A/\mmm \oplus \mmm/\mmm^2 \oplus \dots$, since applying the ker-coker exact sequence to the diagram
\[
\begin{tikzcd}
0  \arrow[r] & \nnn_{n+1} \arrow[d] \arrow[r] & \hat{A} \arrow[d, equal] \arrow[r] & A/\mmm^{n+1} \arrow[d] \arrow[r] & 0 \\
0  \arrow[r] & \nnn_{n} \arrow[r] & \hat{A}  \arrow[r] & A/\mmm^{n}  \arrow[r] & 0
\end{tikzcd}
\]
we obtain
\[ \nnn_n/\nnn_{n+1} = \kerr(A/\mmm^{n+1} \to A/\mmm^{n}) = \mmm^n/\mmm^{n+1}. \]
Now we argue as in the proof of Lemma \ref{lema.domain}:  ${\rm Gr}_{\mmm}(A)$ is a domain, then so is ${\rm gr}(\hat{A})$, and then $\hat{A}$ is a domain.
\end{proof}

\begin{cor}\label{cor.domain}
Let $\locala$ be a local ring formally regular and separated with the $\mmm$-adic topology. Then $A$ is a domain. \qed
\end{cor}

\begin{rem} The hypothesis that $A$ is separated is necessary. In fact, there exist local rings whose maximal ideal is a non-trivial nil ideal (that is, every element is nilpotent) and idempotent (and so formally regular). For instance \cite[\S9, Exercice 1]{BoA8} it is easy to check that the group algebra $\mathbb{F}_p[\mu_{p^\infty}]$ is an example of such a ring, where $\mu_{p^\infty}$ is the group of $p$-power complex roots of unity (its maximal ideal consists on the elements $\sum n_i\zeta_i$ with $\sum n_i=0$).
\end{rem}

We will see now that valuation rings and perfectoid rings are formally regular. This is clear if $\locala$ is a valuation ring of rank 1, since then it is a discrete valuation ring or $\mmm = \mmm^2$ and so $A$ is formally regular with the $\mmm$-adic topology. But we will see much stronger results on the regularity of valuation and perfectoid rings.

\begin{defn}\label{exteriormente.regular} \cite{An1974a} \cite{QMIT}
An ideal $\aaa$ of a ring $A$ is \emph{quasi-regular} (resp. \emph{regular}) if the $A/\aaa$-module $\aaa/\aaa^2$ is flat (resp. projective) and the canonical graded homomorphism
\[ \gamma_*: \wedge^*_{A/\aaa} \aaa/\aaa^2 \to \Tor^A_*(A/\aaa,A/\aaa) \]
is an isomorphism. This is equivalent to
\[ \HH_n(A,A/\aaa,M)=0 \quad (\mathrm{resp. \HH^n(A,A/\aaa,M)=0})\]
for all $n\geq 2$ and any $A/\aaa$-module $M$ \cite[Theorem 10.3]{QMIT} \cite[Th\'eor\`eme A]{An1974a} (resp. \cite[Corollary 10.4]{QMIT} \cite[Th\'eor\`eme B]{An1974a}). When $A$ is noetherian, $\aaa$ is quasi-regular if and only if it is regular if and only if it is locally generated by a regular sequence \cite[6.25]{An1974}.

If $\locala$ is a local ring and $\mmm$ is a quasi-regular ideal, then $\HH_2(A,A/\mmm,A/\mmm)=0$ and so $\HH^2(A,A/\mmm,M)=0$ for any $A/\mmm$-module $M$ by \cite[3.21]{An1974}. Therefore $A$ is formally regular with the discrete topology by definition, and also formally regular with the $\mmm$-adic topology by Proposition \ref{prop.formal.regular.topologies}. The converse is not true (Example \ref{example.non.rigid}).
\end{defn}

If $A$ is a ring, $\aaa$ an ideal of $A$ and $M$ an $A/\aaa$-module, there exists a surjective homomorphism
\[ \Tor^A_2(A/\aaa,M) \to \HH_2(A,A/\aaa,M) \]
(see \cite[Theorem 6.16]{QMIT}, \cite[15.8]{An1974}, or for a direct simple proof dualize the one of \cite[2.3.2]{MR}) and then, if $\aaa$ is flat, it is formally regular with the discrete topology. With a different language this was shown by Planas-Vilanova in \cite{Pla2} (see also \cite[Closing Remark]{Pla3}). Using another result of Planas-Vilanova we can give a stronger result:

\begin{thm}\label{thm.flat.regular}
Any flat ideal $\aaa$ of a ring $A$ is quasi-regular.
\end{thm}
\begin{proof}
By flat base change, $\aaa/\aaa^2$ is a flat $A/\aaa$-module. Then by \cite[Lemma 2.1]{Pla1} the homomorphism
\[ \gamma_n: \wedge^n_{A/\aaa} \aaa/\aaa^2 \to \Tor^A_n(A/\aaa,A/\aaa) \]
is injective for all $n$. Since $\gamma_1$ is an isomorphism and $\Tor^A_n(A/\aaa,A/\aaa)=0$ for $n\geq 2$, $\gamma_n$ is an isomorphism for all $n$.
\end{proof}

For the next corollary, we consider the definition of \emph{perfectoid ring} (for a prime that we will always denote by $p$) as in \cite[Definition 3.5]{BMS} (or equivalently in \cite[Definition 3.5]{BIM}).

\begin{cor}\label{valuation.and.perfectoid}
(i) Any ideal of a valuation ring is quasi-regular. In particular, if $\locala$ is a valuation ring, then it is formally regular with the discrete topology.\\
(ii) Any radical ideal of a perfectoid ring containing $p$ is quasi-regular. In particular, any maximal ideal of a perfectoid ring is quasi-regular.
\end{cor}
\begin{proof}
Any ideal of a valuation ring is flat (\cite[VI, \S3, n. 6, Lemma 1]{BoAC57}) so (i) follows from Theorem \ref{thm.flat.regular}. Now let $A$ be a perfectoid ring. By \cite[Lemma 3.7]{BIM}, $A/\mathrm{rad}(p)$ is perfect and $\mathrm{rad}(p)$ is flat over $A$. Let $\mmm$ be a radical ideal of $A$ containing $p$, and $B=A/\mmm$. We have a Jacobi-Zariski exact sequence
\[ ... \to \HH_3(A/\mathrm{rad}(p),B,M) \to \HH_2(A,A/\mathrm{rad}(p),M) \]
\[ \to \HH_2(A,B,M) \to \HH_2(A/\mathrm{rad}(p),B,M) \]
for any $B$-module $M$. Since $A/\mathrm{rad}(p)$ and $B$ are perfect, then
\[ \HH_n(A/\mathrm{rad}(p),B,M)=0 \]
for all $n\geq 0$ (since the Frobenius isomorphisms induce the zero map on the cotangent complex \cite[Lemme 53]{An1988}). By Theorem \ref{thm.flat.regular}
\[ \HH_n(A,A/\mathrm{rad}(p),M)=0 \]
for all $n\geq 2$. Therefore $\HH_n(A,B,M)=0$ for all $n\geq 2$.

For the particular case, note that if $\mmm$ is a maximal ideal of $A$, $p\in \mmm$ since $A$ is $p$-adically separated and complete (\cite[III, \S2, n. 13, Lemme 3(ii)]{BoAC14}).
\end{proof}

\begin{cor}\label{cor.extension.valuation}
Let $A\to B$ be a flat ring homomorphism and assume that any ideal of $B$ is flat (for instance if $B$ is a Prüfer domain, or more particularly a valuation ring). Then
\[ \HH_n(A,B,k(\qqq))=0 \]
for and prime ideal $\qqq$ of $B$ and all $n\geq 2$.
\end{cor}
\begin{proof}
Let $\ppp=f^{-1}(\qqq)$. By flat base change \cite[4.54]{An1974} we have
\[\tag{1}  \HH_n(A,B,k(\qqq))=\HH_n(A/\ppp,B/\ppp B,k(\qqq)). \]
By the Jacobi-Zariski exact sequence associated to $B\to B/\ppp B \to B/\qqq$ and Theorem \ref{thm.flat.regular} we have
\[ \HH_n(B/\ppp B,B/\qqq ,k(\qqq))=0  \]
for all $n\geq 3$, and then by \cite[5.27]{An1974}
\[\tag{2} \HH_n((B/\ppp B)_{\qqq},k(\qqq) ,k(\qqq))=0  \]
for all $n\geq 3$.
Now, from (2), \cite[7.4]{An1974} and the Jacobi-Zariski exact sequence associated to $k(\ppp) \to (B/\ppp B)_{\qqq} \to k(\qqq)$ we obtain
\[ \HH_n(k(\ppp),(B/\ppp B)_{\qqq},k(\qqq))=0 \]
for all $n\geq 2$
and therefore
\[ \HH_n(A/\ppp,B/\ppp B,k(\qqq))=0 \]
for all $n\geq 2$. Now (1) gives the result.
\end{proof}

\begin{rem}
If $A \to B$ is an extension of valuation rings we can say more. In \cite[Theorem 6.5.12(i)]{GR} it is proved that $\HH_n(A,B,B)=0$ for all $n\geq 2$ and $\HH_1(A,B,B)$ is a flat $B$-module, which is equivalent to $\HH_n(A,B,M)=0$ for all $n\geq 2$ and any $B$-module $M$ (it suffices to apply the universal coefficient spectral sequence
\[ E^2_{pq}=\Tor^B_p(\HH_q(A,B,B),M)\Rightarrow \HH_{p+q}(A,B,M) \]
keeping in mind that $\Tor^B_n(-,-)=0$ if $n\geq 2$).
\end{rem}

Now consider the case of a homomorphism of perfectoid rings $A\to B$. In \cite[Lemma 3.14]{BMS} it is proved that $\HH_n(\mathbb{L}_{B|A}\overset{\rm L}\otimes_\mathbb{Z}\mathbb{F}_p)=0$ for all $n$. We also have:

\begin{cor}\label{perfectoid}
Let $f:A \to B$ a homomorphism of perfectoid rings. Then
\[ \HH_n(A,B,B/\qqq)=0 \]
for any radical ideal $\qqq$ of $B$ containing $p$ and all $n\geq 2$.
\end{cor}
\begin{proof}
From the Jacobi-Zariski exact sequence associated to $A \to B \to B/\qqq$ and Corollary \ref{valuation.and.perfectoid}(ii) we obtain
\[ \HH_n(A,B,B/\qqq)=\HH_n(A,B/\qqq,B/\qqq) \]
for all $n\geq 2$. Let $\ppp=f^{-1}(\qqq)$. Since $A/\ppp$ and $B/\qqq$ are perfect, $\HH_n(A/\ppp,B/\qqq,B/\qqq)=0$ for all $n$, and then from the Jacobi-Zariski exact sequence associated to $A \to A/\ppp \to B/\qqq$ we deduce
\[ \HH_n(A,B/\qqq,B/\qqq) = \HH_n(A,A/\ppp,B/\qqq) \]
for all $n$. But $\HH_n(A,A/\ppp,B/\qqq)=0$ for $n\geq 2$ by Corollary \ref{valuation.and.perfectoid}(ii).
\end{proof}

In the proof of Corollary \ref{valuation.and.perfectoid} we have used that if $f:A \to B$ is a homomorphism of perfect rings of characteristic $p$ and $M$ a $B$-module, then $\HH_n(A,B,M)=0$ for all $n$. More generally, Gabber and Ramero \cite[6.5.13(i)]{GR}, generalizing \cite[Exposé XV, \S1, Proposition 2 (c) (2)]{SGA5}, show that if $f:A \to B$ is a homomorphism of rings of characteristic $p$ such that $\Tor^A_n({^\phi A},B)=0$ for all $n>0$ and the relative Frobenius homomorphism ${^\phi A}\otimes_A B \to {^\phi B}$ is an isomorphism (${^\phi A}$ is the ring $A$ considered as $A$-module via the Frobenius homomorphism $\phi$, and similarly ${^\phi B}$), then $\HH_n(A,B,M)=0$ for all $n$ (see also \cite[Proposition 5.13]{SchPS}). We can give a more precise version of their result as follows:

\begin{prop}\label{precision}
Let $f:A \to B$ be a homomorphism of rings of characteristic $p$ and $n$ an integer. Assume that $\HH_n({^\phi A}\otimes_A B,{^\phi B},M)=0$ for any ${^\phi B}$-module $M$.\\
(i) If $n\in \{0,1\}$ then $\HH_n(A,B,M)=0$ for any $B$-module $M$.\\
(ii) If $n\geq 2$ and $\Tor^A_i({^\phi A},B)=0$ for all $i=1,...,n$, then $\HH_n(A,B,M)=0$ for any $B$-module $M$.\\
(iii) If $n=1$, then for any prime ideal $\qqq$ of $B$, $\HH_0(A,B,k(\qqq))=0$ where $k(\qqq)$ is the residue field of $B$ at $\qqq$.\\
(iv) If $n\geq 2$ and $\Tor^A_i({^\phi A},B)=0$ for all $i=1,...,n-1$, then for any prime ideal $\qqq$ of $B$, $\HH_{n-1}(A,B,k(\qqq))=0$.
\end{prop}
\begin{proof}
Let $M$ be a ${^\phi B}$-module. We have a commutative diagram with exact upper row
\[
\begin{tikzpicture}[baseline= (a).base]
\node[scale=0.78] (a) at (0,0){
\begin{tikzcd}
\HH_{n+1}({^\phi A}\otimes_A B,{^\phi B},M)  \arrow[r, "\delta_{n+1}"] & \HH_n({^\phi A},{^\phi A}\otimes_A B,M) \arrow[r, "\alpha_n"] &  \HH_n({^\phi A},{^\phi B},M) \arrow[r, "\beta_n"] & \HH_n({^\phi A}\otimes_A B,{^\phi B},M) \\
& \HH_n(A,B,M) \arrow[u, "\sigma_n"] \arrow[ru, "\Phi_n"] &  &
\end{tikzcd}
};
\end{tikzpicture}
\]
The homomorphism $\Phi_n$ is zero for all $n$ since it is induced by Frobenius. The homomorphism $\sigma_n$ is an isomorphism for $n=0$, an epimorphism for $n=1$ \cite[2.6.2]{MR}, and an isomorphism for $n\geq 2$ whenever $\Tor^A_i({^\phi A},B)=0$ for all $i=1,...,n$ \cite[4.54]{An1974}. Therefore $\alpha_n=0$ in each of these cases and so $\HH_n({^\phi A}\otimes_A B,{^\phi B},M)=0$ for any ${^\phi B}$-module $M$ implies $\HH_n({^\phi A},{^\phi B},M)=0$ for any ${^\phi B}$-module $M$, that is $\HH_n(A,B,M)=0$ for any $B$-module $M$.

By the same diagram, $\HH_{n+1}({^\phi A}\otimes_A B,{^\phi B},M)=0$ for any $B$-module $M$ implies $\HH_n({^\phi A},{^\phi A}\otimes_A B,M)=0$ and then $\HH_n(A,B,M)=0$ for any ${^\phi B}$-module $M$ if $n=0$ or $\Tor^A_i({^\phi A},B)=0$ for all $i=1,...,n$. In particular, $\HH_n(A,B,k(\qqq))\otimes_{k(\qqq)}{^\phi k(\qqq)}=\HH_n(A,B,{^\phi k(\qqq)})=0$ and then $\HH_n(A,B,k(\qqq))=0$.
\end{proof}

\begin{rem}
In relation with Corollary \ref{cor.extension.valuation} and (iii) and (iv) of Proposition \ref{precision}, note the following result by M. Andr\'e when $B$ is noetherian \cite[Supplément, Proposition 29]{An1974}: if $n$ is an integer such that $\HH_i(A,B,k(\qqq))=0$ for any prime ideal $\qqq$ of $B$ and any $i\geq n$, then $\HH_i(A,B,M)=0$ for any $B$-module $M$ and any $i\geq n$.

\end{rem}

One technical advantage of the $\aaa$-adic topology for an ideal $\aaa$ is the following rigidity result. We will see in Example \ref{example.non.rigid} that it does not hold with the discrete topology.

\begin{prop}\label{prop.f.r.rigid}
Let $A$ be a ring, $\aaa$ and ideal of $A$. If $\aaa$ is a formally regular ideal with the $\aaa$-adic topology then
\[ \colimn \HH^t(A/\aaa^n,A/\aaa,M) = 0 \]
for all $t\geq 2$ and all $A/\aaa$-modules $M$.
\end{prop}
\begin{proof}
This is \cite[12.11]{An1974}, keeping in mind Proposition \ref{prop.regular.projective}.
\end{proof}

\begin{prop}
(i) Let $\aaa$ be an ideal of an ring $A$. If $\aaa$ is formally regular with the $\aaa$-adic (resp. discrete) topology, then $\aaa +(X)$ is a formally regular ideal of $A[X]$ with the $\aaa +(X)$-adic (resp.discrete) topology.\\
(ii) Let $(A_i,\mmm_i,k_i)$ be a direct system of local rings and local homomorphisms, $A:= \underset{i}{\varinjlim} \; A_i$ . If each $A_i$ is a local ring formally regular with the adic topology of its maximal ideal (resp. discrete topology), then so is $A$.
\end{prop}
\begin{proof}
(i) Let $\aaa_n=\aaa^n$ or $\aaa_n=0$ for all $n$. For any ring $B$, $\HH^n(B,B[X],M)=0$ for any $n\geq 1$ and any $B[X]$-module $M$ \cite[3.36]{An1974}, so from the Jacobi-Zariski exact sequence associated to $B\to B[X]\to B$ we deduce $\HH^n(B[X],B,M)=0$ for any $n\geq 2$ and any $B$-module $M$. Therefore the last term in the exact sequence
\[ 0=\colimn \HH^2(A/\aaa_n,A/\aaa,M) \to \colimn \HH^2((A/\aaa_n)[X],A/\aaa,M) \]
\[ \to \colimn \HH^2((A/\aaa_n)[X],A/\aaa_n,M)   \]
vanishes and we deduce
\[ \colimn \HH^2((A/\aaa_n)[X],A/\aaa,M)=0. \]
for any $A/\aaa$-module $M$. The extreme terms in the exact sequence
\[  \colimn \HH^0(A[X],(A/\aaa_n)[X],M) \to \colimn \HH^1((A/\aaa_n)[X],A[X]/I_n,M) \]
\[  \to \colimn \HH^1(A[X],A[X]/I_n,M)  \]
also vanish by \cite[2.3.4]{MR}, where $I_n=(\aaa + (X))^n$ if $\aaa_n=\aaa^n$, or $I_n=(0)$ if $\aaa_n=(0)$. We obtain
\[ \colimn \HH^1((A/\aaa_n)[X],A[X]/I_n,M)=0 \]
for any $A/\aaa$-module $M$.

Finally,the exact sequence
\[ 0= \colimn \HH^1((A/\aaa_n)[X],A[X]/I_n,M) \to \colimn \HH^2(A[X]/I_n,A/\aaa,M)\]
\[ \to \colimn \HH^2((A/\aaa_n)[X],A/\aaa,M)=0  \]
gives us
\[ \colimn \HH^2(A[X]/I_n,A/\aaa,M)=0 \]
as desired.\\

(ii) The ring $A$ is local, since $\mmm :=\underset{i}{\varinjlim} \; \mmm_i$ is a proper ideal of $A$ and any element of $A-\mmm$ comes from some $A_i-\mmm_i$ and so it is a unit. Let $k=A/\mmm$. If each $A_i$ is formally regular with the discrete topology, then $\HH^2(A_i,k_i,k)=0$ for each $i$. By \cite[3.21, Appendice Lemme 43]{An1974} we have
\[ \HH^2(A,k,k)=\HHom_k(\HH_2(A,k,k),k)=  \HHom_k(\underset{i}{\varinjlim} \; \HH_2(A_i,k_i,k),k)\]
\[ = \underset{i}{\varprojlim} \;  \HHom_k(\HH_2(A_i,k_i,k),k) = \underset{i}{\varprojlim} \; \HH^2(A_i,k_i,k) = 0. \]

If each $A_i$ is formally regular with the $\mmm_i$-adic topology, then for each $i$ the homomorphism
\[ \HH^2(A_i/\mmm_i^n,k_i,k) \to  \HH^2(A_i/\mmm_i^{n+1},k_i,k)  \]
vanishes for any $n\neq 2$ as we have seen in Remark \ref{provanishing}.
By \cite[3.21]{An1974}, this homomorphism can be identified to
\[ \HHom_{k_i}(\HH_2(A_i/\mmm_i^n,k_i,k),k) \to \HHom_{k_i}(\HH_2(A_i/\mmm_i^{n+1},k_i,k),k) \]
and since $k_i$ is a field that means that
\[  \HH_2(A_i/\mmm_i^{n+1},k_i,k) \to \HH_2(A_i/\mmm_i^n,k_i,k)  \]
vanishes. Taking limits and using \cite[Appendice Lemme 43]{An1974}, the homomorphism
\[  \HH_2(A/\mmm^{n+1},k,k) \to \HH_2(A/\mmm^n,k,k)  \]
is zero and then applying $\HHom_k(-,k)$ we obtain that
\[ \HH^2(A/\mmm^n,k,k) \to  \HH^2(A/\mmm^{n+1},k,k)  \]
vanishes. We deduce that $A$ is formally regular with the $\mmm$-adic topology.
\end{proof}

\hfill \break
\subsection{Formal regularity and formal smoothness}
\hfill \break

\begin{thm}\label{teorem.f.smooth.regular}
Let $f:\ringa \to \ringb$ be a continuous homomorphism. Assume for the sake of simplicity (see Remark \ref{independent}) that $f(\aaa_n) \subset \bbb_n$ for all $n$.\\
(i) Assume that $\HH^2(A/\aaa,B/\bbb,M)=0$ for any $B/\bbb$-module $M$. If $\aaa$ is formally regular with the $\AAA$ topology and $f$ is formally smooth, then $\bbb$ is formally regular with the $\BBB$ topology.\\
(ii) Assume that the induced homomorphism $\ringa \to (B/\bbb,(0), \{(0)\}_{n>0})$ is formally smooth. If $\bbb$ is formally regular with the $\BBB$ topology then $f$ is formally smooth.
\end{thm}
\begin{proof}
(i) Let $M$ be a $B/\bbb$-module. The exact sequence
\[ 0=\HH^2(A/\aaa,B/\bbb,M) \to \colimn \HH^2(A/\aaa_n,B/\bbb,M) \to \colimn \HH^2(A/\aaa_n,A/\aaa,M)=0  \]
shows that
\[ \colimn \HH^2(A/\aaa_n,B/\bbb,M)=0. \]
Now, the exact sequence (using Proposition \ref{formallysmoothideal})
\[ 0= \colimn \HH^1(A/\aaa_n,B/\bbb_n,M) \to \colimn \HH^2(B/\bbb_n,B/\bbb,M) \to \colimn \HH^2(A/\aaa_n,B/\bbb,M)=0 \]
gives
\[ \colimn \HH^2(B/\bbb_n,B/\bbb,M)=0 \]
as required.\\
(ii) It follows from the exact sequence for each $B/\bbb$-module $M$
\[ 0=\colimn \HH^1(A/\aaa_n,B/\bbb,M) \to \colimn \HH^1(A/\aaa_n,B/\bbb_n,M) \to \colimn \HH^2(B/\bbb_n,B/\bbb,M)=0. \]
\end{proof}

\begin{cor}\label{cor.1.f.smooth.regular}
(i) Let $f:\locala \to \localb$ be a local homomorphism of local rings. If $A$ is formally regular with the $\mmm$-adic topology (resp. the discrete topology) and $f$ formally smooth for the ideal $\nnn$ with the $\nnn$-adic topology (resp. the discrete topology), then $B$ is formally regular with the $\nnn$-adic topology (resp. the discrete topology).\\
(ii) Let $k$ be a field, $\localb$ a local ring, $\BBB$ a filtration on $\nnn$ and $f:k \to (B,\nnn,\BBB)$ a ring homomorphism. If $l|k$ is separable and $B$ is formally regular with the $\BBB$ topology, then $f$ is formally smooth.
\end{cor}
\begin{proof}
It follows from Theorem \ref{teorem.f.smooth.regular}. For (i) note that $\HH^2(A/\mmm,B/\nnn,-)=0$ by \cite[7.4]{An1974}.
\end{proof}

\begin{cor}\label{cor.2.f.smooth.regular}
Let $\locala \to \localb$ be a local homomorphism of local rings. If $A$ is formally regular with the $\mmm$-adic topology (for instance if $A=k$ is a field) and $B$ is a formally smooth $A$-algebra for the ideal $\nnn$ with the $\nnn$-adic topology, then
\[ \colimn \HH^t(A/\mmm^n,B/\nnn^n,M) = 0 \]
for all $t\geq 1$ and all $l$-modules $M$.
\end{cor}
\begin{proof}
Let $t\geq 2$. From the exact sequence
\[ \HH^t(k,l,M) \to \colimn \HH^t(A/\mmm^n,l,M) \to \colimn \HH^t(A/\mmm^n,k,M)=0 \]
and Proposition \ref{prop.f.r.rigid} we obtain
\[  \colimn \HH^t(A/\mmm^n,l,M)=0 \]
since $\HH^t(k,l,M)=0$ by \cite[7.4]{An1974}.
Now the result follows from the exact sequence
\[ 0=\colimn \HH^t(A/\mmm^n,l,M) \to \colimn \HH^t(A/\mmm^n,B/\nnn^n,M) \to \colimn \HH^{t+1}(B/\nnn^n,l,M) \]
where the right term vanishes by Corollary \ref{cor.1.f.smooth.regular}(i) and Proposition \ref{prop.f.r.rigid}.
\end{proof}

\begin{rem}
(i) \emph{Completion.} Let $\locala$ be a local ring, $(\hat{A},\hat{\mmm},k)$ its $\mmm$-adic completion. From Proposition \ref{btildeb} (or Proposition \ref{formallysmoothideal}) we see that $\hat{A}$ is formally smooth over $A$ for its maximal ideal $\hat{\mmm}$ with the limit topology. If $A$ is noetherian then $\hat{A}$ is flat over $A$ and so it is even formally smooth for its maximal ideal with the discrete topology, since by flat base change \cite[4.54]{An1974}
\[ \HH^1(A,\hat{A},k)=\HH^1(k,k,k)=0. \]
We will see another case where this is true:

Let $\locala$ be a local ring, $(\hat{A},\hat{\mmm},k)$ its $\mmm$-adic completion. Assume that $(\hat{A},\hat{\mmm},k)$ is formally regular (as a local ring) with the discrete topology, (that is, $\HH^2(\hat{A},k,k)=0$). Then $\hat{A}$ is formally smooth over $A$ for its maximal ideal with the discrete topology.

In order to see this, we will show that for any local ring $\locala$ the canonical homomorphism
\[ \HH^1(A,\hat{A},k) \to \HH^2(\hat{A},k,k) \]
is injective. From the Jacobi-Zariski exact sequence
\[ \HH^1(\hat{A},k,k) \xrightarrow{\alpha} \HH^1(A,k,k) \to \HH^1(A,\hat{A},k) \to \HH^2(\hat{A},k,k) \]
we have to show that $\alpha$ is surjective. Consider the following diagram with exact upper row
\[
\begin{tikzpicture}[baseline= (a).base]
\node[scale=0.90] (a) at (0,0){
\begin{tikzcd}
\colimn \HH^0(A,A/\mmm^n,k) \arrow[r] & \colimn \HH^1(A/\mmm^n,k,k) \arrow[r, "\beta"] & \HH^1(A,k,k) \arrow[r] & \colimn \HH^1(A,A/\mmm^n,k) \\
 & \colimn \HH^1(\hat{A}/\widehat{\mmm^n},k,k) \arrow["\simeq", u] \arrow[r] & \HH^1(\hat{A},k,k) \arrow[u, "\alpha"] &
\end{tikzcd}
};
\end{tikzpicture}
\]
We have
\[ \colimn \HH^0(A,A/\mmm^n,k)=0=\colimn \HH^1(A,A/\mmm^n,k), \]
then $\beta$ is an isomorphism, and so $\alpha$ is surjective.\\

(ii) \emph{A henselian cover.}
Let $\locala$ be a local ring, and put
\[ A=\colimi A_i \]
where $\{A_i\}$ is the direct system of all local $\mathbb{Z}$-subalgebras of $A$ essentially of finite type with local inclusions, and let
\[ \check{A}:= \colimi \widehat{A_i} \]
where $\widehat{A_i}$ is the completion of $A_i$ with the adic topology of its maximal ideal $\mmm_i$. The homomorphisms $A_i \to \widehat{A_i}$ induce a map
\[ \iota: A \to \check{A} \]
and any local homomorphism $A\to B$ induces a local homomorphism $\check{A} \to \check{B}$. We have:\\
(a) $\check{A}$ is a local henselian ring \cite[I, \S 2, Proposition 1]{Ray}, with maximal ideal $\mmm \check{A}$ (since each $A_i$ is noetherian and then $\mmm_i\hat{A_i}$ is the maximal ideal of $\hat{A_i}$) and residue field $k$. \\
(b) $\iota$ is flat (each $A_i$ is noetherian and then each $A_i \to \widehat{A_i}$ is flat).\\
(c) Since $\iota$ is flat and $\check{A}\otimes_Ak=k$, by flat base change we have isomorphisms
\[ \HH^n(A,k,k)=\HH^n(\check{A},k,k) \]
\[ \HH_n(A,k,k)=\HH_n(\check{A},k,k) \]
for each $n\geq 0$. In particular, $\HH^1(A,\check{A},k)=0$ and then $\check{A}$ is formally smooth over $A$ for its maximal ideal with the discrete topology.\\
(d) $A$ is formally regular with the adic topology of its maximal ideal (resp. the discrete topology) if and only if $\check{A}$ is for its own (with the discrete topology it follows from (c), and with the adic topology from the analogous base change isomorphisms $\HH^2(A/\mmm^n,k,k)=\HH^2(\check{A}/(\mmm \check{A})^n,k,k)$).
\end{rem}

\begin{ex}\label{example.non.rigid}
By an example of Planas-Vilanova \cite{Pla1}, then generalized by M. Andr\'e in \cite{AnExamples}, for any field $k$ and any integer $t\geq 2$, there exists a local ring $\locala$ with residue field $k$ such that $\HH_n(A,k,k)=0$ for all $0\leq n\leq t$ and $\HH_{t+1}(A,k,k)\neq 0$. That means that considering in $A$ the $\mmm$-adic topology and in $k$ the discrete one, $k$ is formally smooth over $A$. However $\HH^{t+1}(A,k,k)\neq 0$ \cite[3.21]{An1974}, so that Corollary \ref{cor.2.f.smooth.regular} does not hold without the hypothesis $A$ formally regular. This is also an example of a local ring $A$ formally regular for the discrete topology but whose maximal ideal is not quasi-regular.
\end{ex}

\begin{rem}

A ring homomorphism $f:A\to B$ is called \emph{absolutely flat} \cite{Fer1} or \emph{weakly étale} \cite{GR}, \cite{Proetale}, if $f$ and the diagonal $B\otimes_A B\to B$ are flat. These homomorphisms have been recently characterised in \cite{dJO} as the ones verifying the following henselian unique lifting property: for any $A$-algebra $R$ and any homomorphism of $A$-algebras $q:B\to R/\ccc$ where $(R,\ccc)$ is a henselian pair, there exists a unique homomorphism of $A$-algebras $B\to R$ lifting $q$. Therefore in analogy with the classical case \cite[$0_{\rm IV}$.19.3.1, $0_{\rm IV}$.19.10.2]{EGAIV1} we can give the associated notion of smoothness and we could ask for the related notion of regularity.

So we say that a ring homomorphism $f:A\to B$ is \emph{w-smooth} if for any $A$-algebra $R$ and any homomorphism of $A$-algebras $q:B\to R/\ccc$ where $(R,\ccc)$ is a henselian pair, there exists (at least) a homomorphism of $A$-algebras $B\to R$ lifting $q$.

This notion of smoothness is stronger (strictly stronger, as we will see below) than the ones defined previously: from Proposition \ref{btildeb} we see that w-smooth implies formally smooth with the discrete topology in $B$ and then by Example \ref{exampleformal} also for any adic topology. From the definition we see that w-smooth algebras are retracts of henselizations of polynomial algebras, and in particular they are flat. More precisely, $f:A\to B$ is w-smooth if and only if there exists a polynomial $A$-algebra $A[X]$ on a set $X$ and an ideal $\aaa$ of $A[X]$ such that $B=A[X]/\aaa$ and the canonical homomorphism of $A$-algebras $A[X]^h\to B$ has a section, where $A[X]^h$ is the henselization of $A[X]$ with respect to $\aaa$ (from Proposition \ref{btildeb} we have a similar characterization of formally smooth algebras with the discrete topology, replacing henselization by completion).

If $f:A\to B$ is w-smooth then $\HH_n(A,B,M)=0=\HH^n(A,B,M)$ for any $B$-module $M$ and all $n\geq 1$, but the reciprocal does not hold (and in particular formal smoothness for the discrete topology does not imply w-smoothness). In order to see this, since w-smooth algebras are retracts of henselizations of polynomial algebras, it suffices to show that $\HH^n(A,A[X]^h,M)=0$, and by the Jacobi-Zariski exact sequence and \cite[3.36]{An1974} it suffices to prove that $\HH^*(R,R^h,-)=0$ for any ring $R$. Since $R\to R^h$ is an inductive limit of \'etale maps, it is weakly \'etale and then by flat base change, the Jacobi-Zarisky exact sequence associated to $R^h\to R^h\otimes_RR^h\to R^h$, and flat base change again
\[ \HH^n(R,R^h,M)=\HH^n(R^h,R^h\otimes_RR^h,M)= \]
\[ \HH^{n+1}(R^h\otimes_RR^h,R^h,M)=\HH^{n+1}(R^h,R^h,M)=0 \]
for any $R^h$-module $M$. For the reciprocal, it suffices to take the canonical map of a perfect local ring $A$ onto its residue field $B$. Since these rings are perfect, $\mathbb{L}_{B|A}\simeq0$, but if $A\neq B$, $B$ is not flat over $A$ and so it is not w-smooth.

Note also the following example \cite[Remark, p. 80]{MR1}. Let $A$ be a perfect field, $E$ a rational function field over $A$ on infinitely many variables and consider the perfect closure $B:=E^{1/p^\infty}$. We have $\mathbb{L}_{B|A}\simeq0$ as before, but by \cite[Theorem 3.1]{MR1}, $\mathrm{fd}_{B\otimes_A B} (B) = \infty$ (flat dimension).

We do not know how to handle the notion of w-smoothness and even less how to define a notion of regularity which is related to w-smoothness in a way similar to Theorem \ref{teorem.f.smooth.regular}, though we can say some things that we should expect. First, we will need a version of Cohen theorems for completions of local rings. Since these completions are not necessarily complete as local rings, the usual sources (for instance \cite[IX]{BoAC89}) do not apply, and we prefer to give the complete details.

\begin{sublemma}\label{lem.surjective}
Let $A$ be a ring, $\aaa$ and ideal of $A$, $f:M\to N$ an $A$-module homomorphism such that $\bar{f}:M/\aaa M \to N/\aaa N$ is surjective. Then $\hat{f}: \hat{M} \to \hat{N}$ is surjective, where we are completing with the $\aaa$-adic topologies in $M$ and in $N$.
\end{sublemma}
\begin{proof}
Let $P:={\rm Im}(f)$. We will see that $\hat{M}\to \hat{P}$ and $\hat{P}\to \hat{N}$ are surjective, considering in $P$ also the $\aaa$-adic topology.

Let $T:=\kerr(f)$, so that we have an exact sequence
\[ 0\to T\to M\to P\to 0 \]
then a Mittag-Leffler projective system of exact sequences
\[ 0\to T/(T\cap \aaa^nM)\to M/\aaa^nM\to P/\aaa^nP\to 0 \]
and therefore an exact sequence
\[0 \to  \limn T/(T\cap \aaa^nM)\to \hat{M}\to \hat{P} \to 0 \]
showing that $\hat{M}\to \hat{P}$ is surjective.

Let us see that $\hat{P}\to \hat{N}$ is surjective. Let $C:= {\rm coker}(f)$, so that we have an exact sequence
\[ 0\to P\to N\to C\to 0 \]
then a projective system of exact sequences
\[ 0\to P/(P\cap \aaa^nN)\to N/\aaa^nN\to C/\aaa^nC\to 0 \]
and so and exact sequence
\[0 \to  \limn P/(P\cap \aaa^nN)\to \hat{N}\to \hat{C} \to 0. \]
Since $\bar{f}:M/\aaa M \to N/\aaa N$ is surjective, $C/\aaa C=0$, so $\hat{C}=\limn C/\aaa^nC=0$ and then we have an isomorphism
\[  \limn P/(P\cap \aaa^nN) = \hat{N}. \]

We also have a projective system of exact sequences
\[ \tag{*} 0 \to (P\cap\aaa^nN)/\aaa^nP \to P/\aaa^nP \to P/(P\cap \aaa^nN) \to 0 \]
which we are going to see that is Mittag-Leffler. Since $M/\aaa M \to N/\aaa N$ is surjective and factorizes as $M/\aaa M \to P/\aaa P \to N/\aaa N$, we deduce that $P/\aaa P \to N/\aaa N$ is surjective, and then $N\subset P+\aaa N$. Then (*) is Mittag-Leffler since
\[ P\cap \aaa^nN \subset P\cap \aaa^n(P+\aaa N) \subset P\cap (\aaa^n P + \aaa^{n+1}N) \subset \aaa^nP + (P\cap \aaa^{n+1}N)  \]
where for the last inclusion note that if $x\in \aaa^nP$, $y\in \aaa^{n+1}N$ are such that $x+y \in P$, then $x\in P$, therefore $y=(x+y)-x \in P$, and so $y\in P\cap \aaa^{n+1}N$, showing that $x+y \in  \aaa^nP + (P\cap \aaa^{n+1}N)$.

Then (*) induces an exact sequence
\[ 0 \to \limn (P\cap\aaa^nN)/\aaa^nP \to \hat{P} \to \limn P/(P\cap \aaa^nN) \to 0 \]
and in particular the homomorphism
\[  \hat{P} \to \limn P/(P\cap \aaa^nN)=\hat{N} \]
is surjective.
\end{proof}

If $f:\locala \to \localb$ is a flat local homomorphism of noetherian local rings such that $B$ is complete and $f(\mmm)B=\nnn$, we will say that $B$ is a Cohen $A$-algebra.

\begin{sublemma}\label{car.p}
Let $\locala$ be a local ring where the characteristic of $k$ is $p>0$, and let $\hat{A}$ be its $\mmm$-adic completion. Let $\mathbb{Z}_{(p)}$ be the localization of $\mathbb{Z}$ at the prime ideal $(p)$. There exist a Cohen $\mathbb{Z}_{(p)}$-algebra $C$ unique up to isomorphism and a local homomorphism $C\to \hat{A}$ inducing an isomorphism on the residue fields.
\end{sublemma}
\begin{proof}
By \cite[3.1.8]{MR} there exists a unique Cohen $\mathbb{Z}_{(p)}$-algebra $(C,\nnn,k)$ which is, by \cite[3.1.2]{MR}, formally smooth over $\mathbb{Z}_{(p)}$ for the ideal $\nnn$ with the $\nnn$-adic topology. By Proposition \ref{btildeb}, there exists a homomorphism of $\mathbb{Z}_{(p)}$-algebras $C\to \hat{A}$ inducing an isomorphism on the residue fields.
\end{proof}

In the situation of Lemma \ref{car.p}, $C$ is a complete discrete valuation ring with maximal ideal $(p)$. The homomorphism $C\to \hat{A}$ is not injective in general, and in this case its kernel is an ideal $(p^n)$, so $\hat{A}$ contains a subring isomorphic to $C/(p^n)$, that is, in the terminology of \cite[IX,\S2, n.2 Définition 2]{BoAC89} $\hat{A}$ contains a ``Cohen subring''.

If $\locala$ is a local ring with $k$ of characteristic zero, then $A$ contains a field. In this case:
\begin{sublemma}\label{car.0}
Let $\locala$ be a local ring containing a field, and let $\hat{A}$ be its $\mmm$-adic completion. Then $\hat{A}$ contains a field $\tilde{k}$ such that the composition map $\tilde{k} \to \hat{A} \to k$ is an isomorphism.
\end{sublemma}
\begin{proof}
Let $F$ be the prime subfield of $A$. Since $k$ is formally smooth over $F$ for the discrete topology (\cite[$0_{\rm IV}$.19.6.1]{EGAIV1}), by Proposition \ref{btildeb} there exists an $F$-algebra homomorphism
\[ h:k\to \hat{A} \]
that composed with the projection map $\hat{A} \to k$ is the identity map. We take $\tilde{k}=h(k)$.
\end{proof}

Now, we are ready to say something about a definition of ``w-regularity''. It should be compatible with the following (indirect and neither easy to handle nor complete) one:

Let $\locala$ be a local ring.\\
(i) If $A$ contains a field, by Lemma \ref{car.0} $\hat{A}$ contains a field $\tilde{k}$ such that the composition map $\tilde{k} \to \hat{A} \to k$ is an isomorphism. We say that $A$ is \emph{w-regular} if $\tilde{k} \to \hat{A}$ is w-smooth.\\
(ii) If $A$ does not contain a field, let $p$ be the characteristic of $k$. By Lemma \ref{car.p} there exists a unique Cohen $\mathbb{Z}_{(p)}$-algebra $C$ with a local homomorphism $C\to \hat{A}$ inducing an isomorphism on the residue fields. Assume that $\hat{A}$ is unramified in the sense that $p$ does not belong to the square of the maximal ideal. Then $A$ is \emph{w-regular} if $C \to \hat{A}$ is w-smooth.

\end{rem}

\hfill \break
\subsection{Regular homomorphisms}
\hfill \break

\begin{defn}\label{43}
We say that a flat ring homomorphism $f:A\to B$ is \emph{regular} (resp. \emph{discretely regular}) if for any prime ideal $\qqq$ of $B$, the $A$-algebra $B_\qqq$ is formally smooth for the ideal $\qqq B_\qqq$ with the $\qqq B_\qqq$-adic (resp. discrete) topology. Discretely regular implies regular (see Example \ref{exampleformal}).

When $B$ is noetherian, there is a usual definition of regular homomorphism (see \cite[IV.6.8.1]{EGAIV2}); we will see in Remark \ref{usual.regular} that our two definitions agree with the usual one in this case.
\end{defn}

\begin{prop}\label{prop.reg.hom}
Let $f:A\to B$ be a flat ring homomorphism. The following are equivalent:\\
(i) $\HH_1(A,B,k(\qqq))=0$ for any prime ideal $\qqq$ of $B$, where $k(\qqq)$ is the residue field of $B$ at $\qqq$.\\
(i') $\HH^1(A,B,k(\qqq))=0$ for any prime ideal $\qqq$ of $B$.\\
(ii) $f$ is discretely regular.\\
(iii) For any prime ideal $\ppp$ of $A$ and any field extension $F|k(\ppp)$, the local ring of $B\otimes_AF$ at any prime ideal is formally regular with the discrete topology.\\
(iv) For any prime ideal $\ppp$ of $A$ and any finite field extension $F|k(\ppp)$, the local ring of $B\otimes_AF$ at any prime ideal is formally regular with the discrete topology.
\end{prop}
\begin{proof}
The proof is standard, but we will give the details. (i) and (i') are equivalent by \cite[3.21]{An1974}. The equivalence of (i') and (ii) follows since $\HH^1(A,B,k(\qqq))=\HH^1(A,B_\qqq,k(\qqq))$ by \cite[5.27]{An1974}.

For (iv) $\Rightarrow$ (iii), let $F|k(\ppp)$ be a field extension and put $F=\colimn F_n$ with $F_n|k(\ppp)$ finite field subextensions of $F|k(\ppp)$. Let $\nnn$ be a prime ideal of $B\otimes_AF$ and $\nnn_n$ its contraction in $B\otimes_AF_n$. We have
\[ \HH_2(B\otimes_AF,k(\nnn),k(\nnn))=\HH_2(B\otimes_A\colimn F_n,k(\nnn),k(\nnn))= \]
\[ \colimn \HH_2(B\otimes_A F_n,k(\nnn),k(\nnn))=  \colimn \HH_2(B\otimes_A F_n,k(\nnn_n),k(\nnn))= \]
\[ \colimn \HH_2(B\otimes_A F_n,k(\nnn_n),k(\nnn_n)\otimes_{k(\nnn_n)}k(\nnn))= \]
\[\colimn \HH_2(B\otimes_A F_n,k(\nnn_n),k(\nnn_n))\otimes_{k(\nnn_n)}k(\nnn)=0 \]
(we have used \cite[5.30, 5.1 and 3.20]{An1974}). This proves (iv) $\Rightarrow$ (iii) and (iii) $\Rightarrow$ (iv) is trivial.

(i) $\Rightarrow$ (iii) Let $\ppp$ be a prime ideal of $A$ and $F|k(\ppp)$ a field extension. Let $\nnn$ be a prime ideal of $B\otimes_AF$ and $\qqq$ its contraction in $B$. Then the contraction of $\qqq$ in $A$ contains $\ppp$. We have
\[ \HH_1(F,B\otimes_AF,k(\nnn))=\HH_1(A,B,k(\nnn))=\HH_1(A,B,k(\qqq))\otimes_{k(\qqq)}k(\nnn)=0 \]
(the first isomorphism  exists since $f$ is flat). Then the Jacobi-Zariski exact sequence
\[  0=\HH_2(F,k(\nnn),k(\nnn)) \to \HH_2(B\otimes_AF,k(\nnn),k(\nnn)) \to \HH_1(F,B\otimes_AF,k(\nnn))=0 \]
gives
\[ \HH_2((B\otimes_AF)_\nnn,k(\nnn),k(\nnn))=\HH_2(B\otimes_AF,k(\nnn),k(\nnn))=0. \]

(iii) $\Rightarrow$ (i) Let $\qqq$ be a prime ideal of $B$ and $\ppp$ its contraction in $A$. Assume first that the characteristic of $k(\ppp)$ is 0. Let $l$ be the residue field of $B_\qqq\otimes_Ak(\ppp)$. By hypothesis, $\HH_2(B\otimes_Ak(\ppp),l,l)=0$ and then from the Jacobi-Zariski exact sequence
\[ \HH_2(B\otimes_Ak(\ppp),l,l) \to \HH_1(k(\ppp),B\otimes_Ak(\ppp),l) \to \HH_1(k(\ppp),l,l)=0  \]
(the last term vanishes since $l|k(\ppp)$ is separable) we obtain
\[ \HH_1(k(\ppp),B\otimes_Ak(\ppp),l)=0.   \]
Then
\[ \HH_1(A,B,k(\qqq))\otimes_{k(\qqq)}l=\HH_1(A,B,l)=\HH_1(k(\ppp),B\otimes_Ak(\ppp),l)=0, \]
and therefore
\[  \HH_1(A,B,k(\qqq))=0. \]
Suppose now that the characteristic of $k(\ppp)$ is $p>0$. By \cite[7.26]{An1974} we have
\[ \HH_1(k(\ppp),B_\qqq\otimes_Ak(\ppp),l)= \HH_2(B_\qqq\otimes_Ak(\ppp)\otimes_{k(\ppp)}{^\phi k(\ppp)},l,l)=  \]
\[ \HH_2(B\otimes_A{^\phi k(\ppp)},l,l)=0  \]
(the superscript $^\phi$ is as in Proposition \ref{precision}) where $l$ is the residue field of the local ring $B_\qqq\otimes_A{^\phi k(\ppp)}$ (note that $B_\qqq\otimes_A{^\phi k(\ppp)}=(B_\qqq\otimes_Ak(\ppp))\otimes_{k(\ppp)}{^\phi k(\ppp)}$ is a local ring and $B_\qqq \to B_\qqq\otimes_A{^\phi k(\ppp)}$ a local homomorphism by \cite[7.18]{An1974}).

We have then
\[ \HH_1(k(\ppp),B\otimes_Ak(\ppp),k(\qqq))\otimes_{k(\qqq)}l= \]
\[ \HH_1(k(\ppp),B\otimes_Ak(\ppp),l)=\HH_1(k(\ppp),B_\qqq\otimes_Ak(\ppp),l)=0 \]
and then
\[\HH_1(A,B,k(\qqq))=\HH_1(k(\ppp),B\otimes_Ak(\ppp),k(\qqq))=0. \]
\end{proof}

\begin{rem}\label{usual.regular}
Let $f:A\to B$ be a ring homomorphism. If $B$ is noetherian, by Corollary \ref{noetheriansmooth}, $f$ is regular if and only if it is discretely regular, and by Corollary \ref{noetherian.regular} and Proposition \ref{prop.reg.hom}(iv) the usual definition of regular homomorphism \cite[IV.6.8.1]{EGAIV2} agrees with these definitions.
\end{rem}

\begin{prop}\label{48}
Let $f:A\to B$ be a regular homomorphism. Then for any prime ideal $\ppp$ of $A$, any field extension $F|k(\ppp)$, and any prime ideal $\nnn$ of $B\otimes_AF$, the local ring $(B\otimes_AF)_\nnn$ is formally regular with the $\nnn$-adic topology.
\end{prop}
\begin{proof}
Let $\ppp$, $\nnn$, and $F$ be as above and let $\qqq$ be the contraction of $\nnn$ in $B$. We have that $B_\qqq$ is a formally smooth $A$-algebra for the ideal $\qqq B_\qqq$ with the $\qqq B_\qqq$-adic topology, and so $B_\qqq\otimes_AF$ is a formally smooth $F$-algebra for the ideal $\qqq B_\qqq\otimes_AF+B_\qqq\otimes_A0=\qqq B_\qqq\otimes_AF$ with the $\qqq B_\qqq\otimes_AF$-adic topology (we can use \cite[$0_{\rm IV}$.19.3.5.(iii)]{EGAIV1} since we are dealing with formal smoothness in the sense of this reference by Corollary \ref{Grothendiecksmooth}). Since $\qqq B_\qqq\otimes_AF\subset \nnn$, $B_\qqq\otimes_AF$ is a formally smooth $F$-algebra for the ideal $\nnn$ with the $\nnn$-adic topology (see Example \ref{exampleformal}). By \cite[$0_{\rm IV}$.19.3.5.(iv)]{EGAIV1}, the local $F$-algebra $(B_\qqq\otimes_AF)_\nnn=(B\otimes_AF)_\nnn$ is formally smooth with the $\tilde{\nnn}=\nnn (B\otimes_AF)_\nnn$-adic topology and then
\[  \colimn \HH^1(F, (B\otimes_AF)_\nnn/\tilde{\nnn}^n,M)=0 \]
for any $l$-module $M$, where $l$ is the residue field of the local ring $(B\otimes_AF)_\nnn$. Then, from the Jacobi-Zariski exact sequence
\[ \colimn \HH^1(F, (B\otimes_AF)_\nnn/\tilde{\nnn}^n,M) \to \colimn \HH^2((B\otimes_AF)_\nnn/\tilde{\nnn}^n,l,M)\to \HH^2(F,l,M)=0 \]
we deduce
\[ \colimn \HH^2((B\otimes_AF)_\nnn/\tilde{\nnn}^n,l,M)=0. \]

\end{proof}

\hfill \break
\subsection{Vanishing of Koszul homology}
\hfill \break

\begin{defn}
Let $\locala$ be a local ring. A set of generators $\{a_i\}_{i\in I}$ of an ideal $\aaa$ is called \emph{minimal} if its images in $\aaa/\mmm\aaa$ form a basis of this $k$-vector space. For instance, by Nakayama's lemma, if $\aaa$ has a finite set of generators, it has a minimal set of generators.
\end{defn}

We denote by $\HH^{Kos}_*(\{a_i\}_{i\in I})$ the Koszul homology associated to the subset $\{a_i\}_{i\in I}\subset A$.

\begin{prop}\label{koszul.1}
Let $\locala$ be a local ring, $\aaa$ an ideal of $A$, $\{a_i\}_{i\in I}$ a set of generators of $\aaa$. If $\HH^{Kos}_1(\{a_i\}_{i\in I})=0$ then $\aaa$ is formally regular with the discrete topology and $\{a_i\}_{i\in I}$ is a minimal set of generators of $\aaa$.
\end{prop}
\begin{proof}
By \cite[2.5.1]{MR} we have an exact sequence for each $A/\aaa$-module $M$
\[ 0  \to \HH_2(A,A/\aaa,M) \to \HH_1^{Kos}(\{a_i\}_{i\in I})\otimes_{A/\aaa}M \to \]
\[F/\aaa F\otimes_{A/\aaa}M  \xrightarrow{\pi (M)} \aaa/\aaa^2\otimes_{A/\aaa}M  \to 0 \]
where $F$ is a free $A$ module with basis $\{X_i\}_{i\in I}$ and $F\to \aaa$ the homomorphism of $A$-modules sending $X_i$ to $a_i$. From this exact sequence and the fact that $\pi (k)$ is an isomorphism if and only if $\{a_i\}_{i\in I}$ is a minimal set of generators of $\aaa$, we deduce the result.
\end{proof}

A similar proof gives:
\begin{lem}\label{koszul.2}
Let $\locala$ be a local ring, $\aaa$ an ideal of $A$. The following are equivalent:\\
(i) $\aaa$ has a minimal set of generators and $\HH_2(A,A/\aaa,k)=0$.\\
(ii) There exists a set of generators $\{a_i\}_{i\in I}$ of $\aaa$ such that $\HH^{Kos}_1(\{a_i\}_{i\in I})\otimes_{A/\aaa}k=0$.

If (ii) holds, $\{a_i\}_{i\in I}$ is a minimal set of generators of $\aaa$ and for any other minimal set of generators $\{b_j\}_{j\in J}$ of $\aaa$, we have $\HH^{Kos}_1(\{b_j\}_{j\in J})\otimes_{A/\aaa}k=0$.
\end{lem}
In particular, if $\mmm$ has a minimal set of generators we deduce that $A$ is formally regular with the discrete topology if and only if $\HH^{Kos}_1(\{a_i\}_{i\in I})=0$ for some (and any) minimal set of generators of $\mmm$.

\begin{prop}\label{koszul.3}
Let $A$ be a ring, $\aaa$ an ideal of $A$, $\{a_i\}_{i\in I}$ a set of generators of $\aaa$. If $\HH^{Kos}_n(\{a_i\}_{i\in I})=0$ for all $n>0$ then $\aaa$ is regular in the sense of Definition \ref{exteriormente.regular}.
\end{prop}
\begin{proof}
Since $\HH^{Kos}_n(\{a_i\}_{i\in I})=0$ for all $n>0$, the module $\HH^{Kos}_1(\{a_i\}_{i\in I})$ is free and the canonical homomorphism
\[  \wedge^*_{A/\aaa}\HH^{Kos}_1(\{a_i\}_{i\in I}) \to \HH^{Kos}_*(\{a_i\}_{i\in I})    \]
is an isomorphism. Then, by \cite{Ro3} \cite[Corollary 3]{BMR2} we have
\[ \HH^n(A,A/\aaa,M)=0   \]
for all $n\geq 3$ and any $A/\aaa$-module $M$. Applying $\HHom_{A/\aaa}(-,M)$ to the complex
\[ U/U_0\to F/\aaa F\to \Omega_{R|A}\otimes_RA/\aaa \]
(notation as in \cite[1.1]{MR}) we obtain an exact sequence
\[ \HHom_{A/\aaa}(F/\aaa F,M) \to \HHom_{A/\aaa}(U/U_0,M) \to \HH^2(A,A/\aaa,M) \to 0, \]
that is, an exact sequence
\[ \HHom_{A/\aaa}(F/\aaa F,M) \to \HHom_{A/\aaa}(\HH^{Kos}_1(\{a_i\}_{i\in I}),M) \to \HH^2(A,A/\aaa,M) \to 0. \]
We deduce
\[  \HH^2(A,A/\aaa,M) = 0, \]
and then
\[ \HH^n(A,A/\aaa,M)=0   \]
for all $n\geq 2$.

\end{proof}

\begin{prop}\label{koszul.4}
Let $\locala$ be a local ring, $\aaa$ an ideal of $A$. If $\aaa$ is regular and has a minimal set of generators, then $\HH^{Kos}_n(\{a_i\}_{i\in I})\otimes_{A/\aaa}k=0$ for any minimal set of generators $\{a_i\}_{i\in I}$ of $\aaa$ and any $n>0$.
\end{prop}
\begin{proof}
By \cite{Ro3} \cite[Corollary 3]{BMR2} the canonical homomorphism
\[  \wedge^*_{A/\aaa}\HH^{Kos}_1(\{a_i\}_{i\in I}) \to \HH^{Kos}_*(\{a_i\}_{i\in I})  \]
is an isomorphism for any set of generators $\{a_i\}_{i\in I}$ of $\aaa$. Taking $\{a_i\}_{i\in I}$ minimal, Lemma \ref{koszul.2} says that $\HH^{Kos}_1(\{a_i\}_{i\in I})\otimes_{A/\aaa}k=0$, and then $\HH^{Kos}_n(\{a_i\}_{i\in I})\otimes_{A/\aaa}k=0$ for all $n>0$.
\end{proof}

\begin{rem}
(i) We do not know if Proposition \ref{koszul.4} holds replacing the hypothesis ``$\aaa$ is regular'' by ``$\aaa$ is quasi-regular'', though if $A$ contains a field, using \cite[Erratum]{Ro3}, the same proof works.\\
(ii) In \cite[Example 1]{Kab} there is an example of an ideal $\aaa$ of a ring $A$ formally regular with the $\aaa$-adic topology and a minimal set of generators $\{a_i\}_{i\in I}$ of $\aaa$ such that $\HH^{Kos}_1(\{a_i\}_{i\in I})\neq 0$. Moreover, in this example $A$ is local and $\aaa$ is its maximal ideal which is finitely generated.
\end{rem}

\section{Descent: raw results}

Let $f:\locala\to \localb$ be a local homomorphism. Denote flat or Tor dimension by fd. The main result in this section is the next one, that will allow us to prove descent results for regularity from $B$ to $A$ when $\mathrm{fd}_A(B)<\infty$:

\begin{thm}\label{main}
If $\Tor^A_n(k,B)=0$ for all $n>>0$ then
\[  \HH_2(A,k,l) \xrightarrow{\alpha} \HH_2(B,l,l)  \]
is injective.
\end{thm}

When $A$ and $B$ are noetherian the result is well known and due to Avramov, who used it to prove that the localizations of a complete intersection ring are complete intersection rings (see \cite{Avr} for his most general version). A different proof can be seen in \cite{ABM}. When the rings are not noetherian, we will need a different method for the proof.\\

\noindent \textit{Proof of Theorem \ref{main}} Let $A[\{X_i\}_{i\in I}]$ be a polynomial $A$-algebra such that there exists a surjective homomorphism of $A$-algebras $\pi:A[\{X_i\}_{i\in I}]\to B$. Let $C=A[\{X_i\}_{i\in I}]_\qqq$ where $\qqq=\pi^{-1}(\nnn)$, and for each finite subset $J$ of $I$, let $A_J=A[\{X_i\}_{i\in J}]_{\nnn_J}$ where $\nnn_J$ is the inverse image of the ideal $\qqq$. Let $k_J$ be the residue field of $A_J$.

We will see first that $\Tor_n^{A_J}(k_J,B)=0$ for all $n>>0$. In the change of rings spectral sequence
\[ E^2_{pq}=\Tor_p^{A_J\otimes_Ak}(\Tor_q^{A_J}(A_J\otimes_Ak,B),k_J)\Rightarrow \Tor_{p+q}^{A_J}(B,k_J) \]
we have $E^2_{pq}=0$ for all $p>>0$ since $A_J\otimes_Ak$ is a localization of a finite type polynomial $k$-algebra, and so a ring of finite homological dimension. We also have $\Tor_q^{A_J}(A_J\otimes_Ak,B)=\Tor^A_q(k,B)$ and so $E^2_{pq}=0$ for all $q>>0$ by hypothesis. Therefore from the spectral sequence we obtain $\Tor_n^{A_J}(B,k_J)=0$ for all $n>>0$.

Consider the following commutative diagram
\[
\begin{tikzcd}
\Tor^C_2(B,l) \arrow[d, "\epsilon"] \arrow[r, "\gamma"] & \HH_2(C,B,l) \arrow[d, "\zeta"] \\
\Tor^C_2(l,l) \arrow[r, "\eta"] & \HH_2(C,l,l)
\end{tikzcd}
\]
where $\gamma$ (and $\eta$) are surjective \cite[proof of 15.4]{An1974}. We are going to see that $\zeta=0$.

Let $\alpha\in\Tor_2^C(B,l)$. It suffices to show that $\epsilon (\alpha)\in \Tor_1^C(l,l)\cdot \Tor_1^C(l,l)$, since $\eta (\Tor_1^C(l,l)\cdot \Tor_1^C(l,l))=0$ by \cite[18.34]{An1974}. Since $\Tor_2^C(B,l)= \underset{J}{\varinjlim} \; \Tor_2^{A_J}(B,k_J)$, there exists some $J$ such that $\alpha$ is the image of some $\beta\in\Tor_2^{A_J}(B,k_J)$. Expand the above diagram as

\[
\begin{tikzcd}
\beta\in\Tor_2^{A_J}(B,k_J) \arrow[d, "\epsilon_J"] \arrow[r] & \Tor^C_2(B,l) \arrow[d, "\epsilon"] \arrow[r, "\gamma"] & \HH_2(C,B,l) \arrow[d, "\zeta"] \\
\Tor_2^{A_J}(l,k_J) \arrow[r] & \Tor^C_2(l,l) \arrow[r, "\eta"] & \HH_2(C,l,l)
\end{tikzcd}
\]
In order to see that $\epsilon (\alpha)\in \Tor_1^C(l,l)\cdot \Tor_1^C(l,l)$, it suffices to see that $\epsilon_J (\beta)\in \Tor_1^{A_J}(l,k_J)\cdot \Tor_1^{A_J}(l,k_J)$.

Let $X$ be a free simplicial resolution of the $A_J$-algebra $k_J$ in the usual sense (\cite[4.30]{An1974}). Let $\beta=[x]\in \Tor^{A_J}_2(B,k_J)$ be represented by a cycle $x\in C(B\otimes_{A_J}X)$ where if $Y$ is a simplicial algebra, $C(Y)$ denotes the DG algebra with $C(Y)_n=Y_n$ and differential $d_n=\sum_{i=0}^n (-1)^i\partial_n^i$. Since $\Tor^{A_J}_{2n}(B,k_J)=0$ for $n>>0$, the class of the cycle $x^{(n)}$ vanishes, where $x^{(n)}$ denotes the $n$th divided power of $x$ in the simplicial sense (see for instance \cite[1.34, 1.35]{Pi}). Since $B\otimes_{A_J}X \to l\otimes_{A_J}X$ is a homomorphism of simplicial rings, $\tilde{\epsilon}: C(B\otimes_{A_J}X) \to C(l\otimes_{A_J}X)$ is a homomorphism of DG algebras with divided powers, and then $0=\epsilon[x^{(n)}]=[\tilde{\epsilon}(x)^{(n)}]$. Since the DG algebra $C(l\otimes_{A_J}X)$ comes from a simplicial algebra, the divided power structure on $C(l\otimes_{A_J}X)$ induces a divided power structure on its homology $\HH(l\otimes_{A_J}X)=\Tor^{A_J}(l,k_J)$ (see \cite[1.36]{Pi}) and then $\epsilon[x]^{(n)}=[\tilde{\epsilon}(x)]^{(n)}=[\tilde{\epsilon}(x)^{(n)}]=0$.

We know that $\Tor^{A_J}(l,k_J)=\Tor^{A_J}(k_J,k_J)\otimes_{k_J}l$ is a Hopf $l$-algebra with divided powers \cite[Th\'eor\`eme 31]{An1970}, and then by \cite[Theorem 1.(a)]{Sj} it is a free $l$-algebra with divided powers. Therefore $\epsilon[x]$ must be contained in the decomposable part of $\Tor^{A_J}_2(l,k_J)$, that is, $\epsilon[x]\in \Tor^{A_J}_1(l,k_J)\cdot \Tor^{A_J}_1(l,k_J)$ as we want to prove, and therefore $\zeta=0$.

Now, from the commutative diagram

\[
\begin{tikzcd}
\HH_2(A,B,l) \arrow[d, "\simeq"] \arrow[r] & \HH_2(A,l,l) \arrow[d, "\simeq"] \\
\HH_2(C,B,l) \arrow[r, "0"] & \HH_2(C,l,l)
\end{tikzcd}
\]
we deduce that $\HH_2(A,B,l) \to \HH_2(A,l,l)$ vanishes, and then
\[  \HH_2(A,k,l)=\HH_2(A,l,l)\to \HH_2(B,l,l)  \]
is injective by the Jacobi-Zariski exact sequence. \qed \\

Note that the above proof gives a more precise result when $A\to B$ is surjective:

\begin{thm}\label{second}
Let $\locala$ be a local ring, $\aaa$ an ideal of $A$, $B=A/\aaa$. If $\Tor_{2n}^A(k,B)=0$ for some $n>0$ then the homomorphism
\[  \HH_2(A,k,k) \xrightarrow{\alpha} \HH_2(B,k,k)  \]
is injective.
\end{thm}
\begin{proof}
The proof is similar to that of Theorem \ref{main}, using directly the commutative diagram
\[
\begin{tikzcd}
\Tor^A_2(B,k) \arrow[d] \arrow[r, "\gamma"] & \HH_2(A,B,k) \arrow[d, "\zeta"] \\
\Tor^A_2(k,k) \arrow[r] & \HH_2(A,k,k)
\end{tikzcd}
\]
where $\gamma$ is surjective, in order to see that $\zeta =0$.

\end{proof}

We end this section with another proof of Theorem \ref{main} when $A$ is noetherian, since it allows to lessen a little the hypothesis in this case. In the commutative diagram of exact sequences \cite[2.5.1]{MR}
\[
\begin{tikzpicture}[baseline= (a).base]
\node[scale=0.90] (a) at (0,0){
\begin{tikzcd}
0  \arrow[r] & \HH_2(A,k,l) \arrow[d, "\alpha"] \arrow[r] & \HH_1^{Kos}(\{x_i\})\otimes_kl \arrow[d, "\beta"] \arrow[r] & F/\mmm F\otimes_kl  \arrow[d] \arrow[r, "\pi"] & \mmm/\mmm^2\otimes_kl  \arrow[d] \arrow[r] & 0 \\
0  \arrow[r] & \HH_2(B,l,l)  \arrow[r] & \HH_1^{Kos}(\{y_i\}) \arrow[r] & G/\nnn G  \arrow[r] & \nnn/\nnn^2 \arrow[r] & 0
\end{tikzcd}
};
\end{tikzpicture}
\]
(where $\{x_i\}_{i\in I}$ a set of generators of $\mmm$, $\{y_j\}_{j\in J}$ a set of generators of $\nnn$, $F$ (resp. $G$) a free $A$-module (resp. $B$-module) with basis $\{X_i\}_{i\in I}$ (resp. $\{Y_j\}_{j\in J}$), and $F\to \mmm$ (resp. $G \to \nnn$) the obvious surjective homomorphism) we can choose as $\{x_i\}$ a \emph{minimal} set of generators of $\mmm$ when $A$ is noetherian. Then $\pi$ is an ismomorphism, so that if (and only if) $\alpha$ is injective, so is $\beta$. Therefore we can work in the context of Koszul complexes and then an elaboration of the ideas of \cite{Avr} suffices for a proof. We will use some definitions and facts on differential graded (anti-)commutative algebras (DG algebras). We refer to the first sections of \cite{GL} for them.

\begin{lem}\label{s+r}
Let $s\geq 0$, $r\geq 1$ be integers. Let $X$ be a DG algebra such that $\HH_n(X)=0$ for all $n\in [s,s+r]$. Let $X'$ be a DG algebra obtained adjoining (to $X$) $r$ variables of degree 1 to kill cycles (in the sense of \cite[I, \S 2]{GL}). Then $\HH_{s+r}(X')=0$.
\end{lem}
\begin{proof}
It is sufficient to show that for the DG algebra $Y=X<S;dS=s>$ obtained adjoining to $X$ one variable of $S$ of degree 1, $\HH_n(Y)=0$ for all $n\in [s+1,s+r]$. This follows from the exact sequence of \cite[p. 19]{GL}
\[  ... \to \HH_{n+1}(X) \xrightarrow{\partial} \HH_{n+1}(X) \to \HH_{n+1}(Y)\to \HH_n(X) \xrightarrow{\partial} \HH_n(X)\to \HH_n(Y) \to ...   \]
\end{proof}

\begin{thm}\label{first}
Let $f:\locala\to \localb$ be a local homomorphism. Assume that $\mmm$ has a minimal set of generators. If one of the following two conditions holds\\
(i) there exists an integer $s$ such that $\Tor_n^A(k,B)=0$ for all $n\geq s$, or\\
(ii) there exist an integer $r\geq 0$, elements $t_1,\dots,t_r \in \nnn$ such that $f(\mmm) B+(t_1,\dots,t_r)=\nnn$ and an integer $s\equiv r \; (2)$ such that $\Tor_n^A(k,B)=0$ for all $n\in [s,s+r]$,

then the homomorphism
\[  \HH_2(A,k,l) \xrightarrow{\alpha} \HH_2(B,l,l)  \]
is injective.
\end{thm}
\begin{proof}
Let $\{u_i\}_{i\in I}$ be a minimal set of generators of the ideal $\mmm$ of $A$. Let $X$ be a minimal DG resolution of the $A$-algebra $k$ \cite[1.6.4]{GL} with 1-skeleton $X^1=A<\{U_i\}_{i\in I};dU_i=u_i>$ (the notation is as in the proof of \cite[1.2.3]{GL}), that is, $X^1$ is the Koszul complex associated to the set of generators $\{u_i\}_{i\in I}$ of $\mmm$.

Let $\{f(u_i)\}_{i\in I}\cup \{t_j\}_{j\in J}$ be a set of generators of the ideal $\nnn$ of $B$ (under the hypothesis (ii) we chose $J=\{1,\dots,r\}$). Consider the associated Koszul complex
\[ Y=B<\{U_i\}_{i\in I}\cup \{T_j\}_{j\in J}; dU_i=f(u_i), dT_j=t_j>.  \]
The homomorphism of DG algebras $\varphi :X^1\to Y$ extending $f$ by $\varphi(U_i)=U_i$ induces a homomorphism on the Koszul homology modules
\[  \beta:\HH^{Kos}_1(\{u_i\}_{i\in I};A)\otimes_kl=H_1(X^1)\otimes_kl \to  \HH_1(Y)=\HH^{Kos}_1(\{f(u_i)\}_{i\in I}\cup \{t_j\}_{j\in J};B). \]
By the commutative diagram at the beginning of this section, we have to see that $\beta$ is injective.

Let $X^2=X^1<\{V_e\}_{e\in E};dV_e=v_e>$ be the 2-skeleton of $X$. Since $X$ is minimal, the homology classes $\{[v_e]\}_{e\in E}$ of the cycles $v_e$ form a $k$-basis of $\HH_1(X^1)=\HH^{Kos}_1(\{u_i\}_{i\in I};A)$. Therefore it is enough to show that the set $\{\beta([v_e]\otimes 1)\}_{e\in E} \subset \HH_1(Y)$ is linearly independent.

On the contrary, suppose that there exists a finite non-empty subset $E_0\subset E$ and non-zero elements $\overline{\lambda_e} \in l$ ($e\in E_0$) such that
\[  \sum_{e\in E_0} \overline{\lambda_e} [\varphi(v_e)]=0. \]
Let $\lambda_e\in B-\nnn$ representants of $\overline{\lambda_e}$. We have then
\[  \sum_{e\in E_0} \lambda_e \varphi(v_e)=dR \]
for some $R\in Y_2$.

Under the hypothesis (i), let $J_0\subset J$ finite such that $R\in B<\{U_i\}_{i\in I}\cup \{T_j\}_{j\in J_0}>$, and under the hypothesis (ii), let $J_0= J$. Let
\[  Z=(X\otimes_AB)\otimes_B(B<\{T_j\}_{j\in J_0};  dT_j=t_j>).  \]
Let $\tilde{R}\in Z$ be the image of $R$ in $Z$ by the canonical homomorphism
\[  B<\{U_i\}_{i\in I}\cup \{T_j\}_{j\in J_0}> \to  (X\otimes_AB)<\{T_j\}_{j\in J_0};  dT_j=t_j>=Z, \]
$V$ the image of
\[  \sum_{e\in E_0} V_e\otimes \lambda_e \in X\otimes_AB \]
in $Z$ by the canonical map $X\otimes_AB \to Z$, and $W=V-\tilde{R}\in Z_2$. We have
\[ dW=  \sum_{e\in E_0}\lambda_e \varphi(v_e)-d\tilde{R} =0, \]
and then
\[ dW^{(h)}=(dW)W^{(h-1)}=0 \]
for all $h\geq 0$, where $W^{(h)}$ denotes the $n$th divided power of $W$ (\cite[I, \S\S 7-8]{GL}). Therefore $W^{(h)}$ is a cycle in $Z_{2h}$ for all $h>0$.

Now we will see that $W^{(h)}$ is not a boundary. The set of finite products of elements $U_i$ ($i\in I$), $T_j$ ($j\in J_0$), $V_e^{(q)}$ ($e\in E$, $q>0$) is part of a basis of $Z$ as free $B$-module, and $W^{(h)}$ is a linear combination of these elements, where one of the summands is $\lambda_e^hV_e^{(h)}$ with $\lambda_e^h \notin \nnn$ (\cite[1.7.1]{GL}). Therefore $W^{(h)}\notin \nnn Z$. Since $X$ is minimal, $dX\subset \mmm X$, and then, since $dT_j=t_j\in\nnn$ for all $j$, $dZ\subset \nnn Z$. From this we deduce that $W^{(h)}$ is not a boundary, and in particular $\HH_{2h}(Z)\neq 0$ for all $h>0$.

Finally, we examine cases (i) and (ii) separately:\\
(i) Since $\HH_n(X\otimes_AB)=\Tor_n^A(k,B)=0$ for all $n\geq s$, from Lemma \ref{s+r} we obtain $\HH_n(Z)=0$ for all $n\geq s + |J_0|$, contradicting that $\HH_{2h}(Z)\neq 0$ for all $h>0$.\\
(ii) Since $\HH_n(X\otimes_AB)=0$ for $n\in [s,s+r]$ and $r=|J_0|$, from Lemma \ref{s+r} we obtain $\HH_{s+r}(Z)=0$. Since $r+s$ is even, we arrive to the same contradiction.
\end{proof}

\section{Descent: processed results}

\hfill \break
\subsection{Descent of formal regularity and descent from perfectoid algebras}
\hfill \break

Theorem \ref{main} gives immediately the following one:

\begin{thm}\label{main3}
Let $\locala \to \localb$ be a local homomorphism of local rings such that $\mathrm{fd}_A(B)<\infty$. If $B$ is formally regular for the discrete topology, then so is $A$. \qed
\end{thm}

The same is true when we have local homomorphisms $A\to B\to C$ with $\mathrm{fd}_A(C)<\infty$ and $B$ formally regular for the discrete topology.

Since valuation and perfectoid rings are examples of formally regular rings, a particular case is the following corollary, that when $A$ is noetherian was obtained in \cite[Corollary 4.8]{BIM}.

\begin{cor}\label{corollaryBIM}
Let $\locala$ be a local ring and $B$ a perfectoid $A$-algebra such that $\mmm B\neq B$. If $\Tor_n^A(k,B)=0$ for all $n>>0$ then $A$ is formally regular for the discrete topology.
\end{cor}
\begin{proof}
Let $\nnn$ be a maximal ideal of $B$ containing $\mmm B$. Then consider the local homomorphism $A\to B_\nnn$, where $B_\nnn$ is formally regular for the discrete topology by Corollary \ref{valuation.and.perfectoid}.
\end{proof}

When $A\to B$ is surjective, we can say a little more (since $\Tor^A_n(k,k)=0$ for all $n>>0$ implies that $A$ is formally regular for the discrete topology as we will see in Proposition \ref{prop.levin}):

\begin{prop}
Let $\locala$ be a local ring, $B=A/I$ for some ideal $I$ of $A$. If $B$ is perfectoid and $\Tor_n^A(B,k)=0$ for all $n>>0$ then $\Tor_n^A(k,k)=0$ for all $n>>0$.
\end{prop}
\begin{proof}
By \cite[Lemma 3.7]{BIM} $\mathrm{rad}(p)$ is a flat ideal of $B$ and $B/\mathrm{rad}(p)$ is perfect, and so its residue field $k$ is perfect. Therefore by \cite[3.16]{Gras}
\[ \Tor_i^{B/\mathrm{rad}(p)}(k,k)=0 \]
for all $i>0$. The $B/\mathrm{rad}(p)$-module structures on $\Tor_*^B(B/\mathrm{rad}(p),k)$ on the right and left factor agree, since $B\to B/\mathrm{rad}(p)$ is surjective. Then, in the change of rings spectral sequence
\[ \mathrm{E}^2_{pq}=\Tor_p^{B/\mathrm{rad}(p)}(\Tor_q^B(B/\mathrm{rad}(p),k),k) \Rightarrow \Tor^B_{p+q}(k,k)  \]
we have $\mathrm{E}^2_{pq}=0$ for $p>0$ (since $\Tor^B_q(B/\mathrm{rad}(p),k)$ is a direct sum of copies of $k$) and for $q>1$ (since $\mathrm{rad}(p)$ is flat). We deduce
\[ \Tor_n^B(k,k)=0 \]
for all $n\geq 2$.

Now we put $A$ in scene. Consider the change of rings spectral sequence
\[ \tilde{\mathrm{E}}^2_{pq}=\Tor_p^B(\Tor_q^A(B,k),k) \Rightarrow \Tor^A_{p+q}(k,k).  \]
Since $A\to B$ is surjective, the $B$-module structures on $\Tor_*^A(B,k)$ given on the right and left factor agree, and then
\[ \Tor_p^B(\Tor_q^A(B,k),k)\simeq \Tor_p^B(k^{X_q},k) \]
where $k^{X_q}$ is a direct sum of copies of $k$. We deduce $\tilde{\mathrm{E}}^2_{pq}=0$ for all $p\geq 2$ and for all $q>>0$. Thus
\[ \Tor_n^A(k,k)=0 \]
for all $n>> 0$ as we wanted to prove.

\end{proof}

We will give now analogues to Corollary \ref{corollaryBIM} for complete intersections and Gorenstein rings, following \cite{Edin}. We will give the details, though the proofs are essentially the same, once we have the results of section 2. We start with a mixture of the definitions of finite complete intersection flat dimension \cite{SW}, \cite{SSY} and finite upper complete intersection dimension \cite{Tak}.

\begin{defn}\label{def.cidim}
Let $\locala$ be a noetherian local ring, $M$ an $A$-module. We say that $M$ has finite upper complete intersection flat dimension if there exists a local flat homomorphism $A\to A'$ of noetherian local rings with regular closed fibre $A'/\mmm A'$ and a surjective homomorphism $Q\to A'$ whose kernel is generated by a regular sequence, where $Q$ is a noetherian local ring, such that $\mathrm{fd}_Q(A'\otimes_AM)<\infty$.
\end{defn}

\begin{prop}\label{prop.cidim}
Let $(A,\mmm,k)$ be a noetherian local ring, $B$ a perfectoid $A$-algebra such that $\mmm B\neq B$. If $B$ has finite upper complete intersection flat dimension over $A$, then $A$ is a complete intersection.
\end{prop}
\begin{proof}
Let $\nnn$ be a maximal ideal of $B$ containing $\mmm B$, $l=B/\nnn$. Let $A\to A'$, $Q\to A'$ be as in Definition \ref{def.cidim}, such that $\mathrm{fd}_Q(A'\otimes_AB)<\infty$. Let $\mmm'$ be the maximal ideal of $A'$, $k'=A'/\mmm'$ its residue field. Let $\nnn'$ be a maximal ideal of $A'\otimes_AB$ containing the images of $\mmm'$ and $\nnn$ in $A'\otimes_AB$ \cite[I.3.2.7.1.(ii)]{EGAI}, and $l'=(A'\otimes_AB)/\nnn'$.

We have a commutative diagram
\[
\begin{tikzcd}
 & \HH_2(B,l,l') \arrow[d, "\delta"] \\
\HH_2(A',A'\otimes_Ak,l') \arrow[d, "\epsilon"] \arrow[r, "\gamma"] & \HH_2(A'\otimes_AB,A'\otimes_Al,l') \arrow[d] \\
\HH_2(A',k',l') \arrow[r, "\lambda"] & \HH_2(A'\otimes_AB,l',l')
\end{tikzcd}
\]
and an exact sequence
\[  \HH_2(A',A'\otimes_Ak,l') \xrightarrow{\epsilon} \HH_2(A',k',l') \to \HH_2(A'\otimes_Ak,k',l')  \]
where the right term vanishes by \cite[6.26]{An1974}, and so $\epsilon$ is surjective. By flat base change, $\delta$ is an isomorphism and, by Corollary \ref{valuation.and.perfectoid}, $\HH_2(B,l,l')=0$. Therefore $\lambda=0$.

By Theorem \ref{main} the composition map
\[  \HH_2(Q,k',l') \to \HH_2(A',k',l') \xrightarrow{\lambda} \HH_2(A'\otimes_AB,l',l') = \HH_2((A'\otimes_AB)_{\nnn '},l',l') \]
is injective, and then $\HH_2(Q,k',l')=0$. By \cite[6.26]{An1974} $Q$ is regular and then $A'$ is a complete intersection. By flat descent \cite[Corollaire 2]{Avr}, $A$ is a complete intersection.
\end{proof}

\begin{defn}\label{def.gorenstein}
Let $\locala$ be a noetherian local ring, $M$ an $A$-module. We say that $M$ has finite upper Gorenstein flat dimension if there exists a local flat homomorphism $A\to A'$ of noetherian local rings with regular closed fibre $A'/\mmm A'$ and a surjective homomorphism of noetherian local rings $Q\to A'$ such that there exists some $n$ such that $\mathrm{Ext}_Q^n(A',Q)=A'$, $\mathrm{Ext}_Q^i(A',Q)=0$ for all $i\neq n$ and $\mathrm{fd}_Q(A'\otimes_AM)<\infty$.
\end{defn}

\begin{prop}\label{Gorens}
Let $(A,\mmm,k)$ be a noetherian local ring, $B$ a perfectoid $A$-algebra such that $\mmm B\neq B$. If $B$ has finite upper Gorenstein flat dimension over $A$, then $A$ is Gorenstein.
\end{prop}
\begin{proof}
Let $A'$, $Q$ be as in Definition \ref{def.gorenstein}. We can see as in the proof of Proposition \ref{prop.cidim} that $Q$ is regular. Then, from the change of rings spectral sequence
\[  \mathrm{E}^{pq}_2= \mathrm{Ext}_{A'}^p(l',\mathrm{Ext}_Q^q(A',Q)) \Rightarrow \mathrm{Ext}_Q^{p+q}(l',Q) \]
where $l'$ is the residue field of $A'$, we deduce that $A'$ is Gorenstein. By flat descent, $A$ is Gorenstein.
\end{proof}

\hfill \break
\subsection{A result of Levin}
\hfill \break

\begin{prop}\label{prop.levin}
Let $\locala$ be a local ring such that $\Tor^A_{2n}(k,k)=0$ for some $n>0$. Then $A$ is formally regular with the discrete topology.
\end{prop}
\begin{proof}
By Theorem \ref{second}, the homomorphism
\[  \HH_2(A,k,k)\to \HH_2(k,k,k)=0  \]
is injective.
\end{proof}

\begin{cor}\label{cor.levin}
Let $\locala$ be a local ring such that $\mmm$ has a minimal set of generators $\{x_i\}_{i\in I}$. If $\mathrm{fd}_A(k)<\infty$ then $\HH_1^{Kos}(\{x_i\}_{i\in I})=0$.
\end{cor}
\begin{proof}
It follows from Proposition \ref{prop.levin} and Lemma \ref{koszul.2}.
\end{proof}

According to \cite[Remark after Theorem 1.1]{Kab}, G. Levin has proved Corollary \ref{cor.levin} when $\mmm$ is finitely generated. In fact, both results are equivalent, since if $\Tor^A_n(k,k)=0$ for all $n>>0$ and $\mmm$ has a minimal set of generators, then $\mmm$ is finitely generated by \cite[Theorem 1]{Nor}.

\hfill \break
\subsection{A result of Rodicio}
\hfill \break

The following result is due to A. G. Rodicio \cite{Ro2} when $A$ and $B$ are noetherian.
\begin{thm}
Let $f:A\to B$ be a flat homomorphism, $\mu: B\otimes_AB\to B$ the multiplication. If $\mathrm{fd}_{B\otimes_AB}(B)<\infty$ then $f$ is discretely regular.
\end{thm}
\begin{proof}
Let $\qqq$ be a prime ideal of $B$. By Proposition \ref{prop.reg.hom} we have to show that $\HH_1(A,B,k(\qqq))=0$. Denote $\nnn:=\mu^{-1}(\qqq)$ and $l=k(\qqq)$ (the residue field of $(B\otimes_AB)_\nnn$ and $B_\qqq$). By Theorem \ref{second} the homomorphism
\[ \alpha_2:\HH_2((B\otimes_AB)_\nnn,l,l)\to \HH_2(B_\qqq,l,l) \]
is injective. Since the composition homomorphism $B\xrightarrow{id\otimes 1} B\otimes_AB \xrightarrow{\mu} B$ is the identity map, the homomorphism
\[ \alpha_n:=\HH_n((B\otimes_AB)_\nnn,l,l)=\HH_n(B\otimes_AB,l,l)\to \HH_n(B,l,l)=\HH_n(B_\qqq,l,l) \]
has a section (and in particular it is surjective) for any $n$. Therefore in the Jacobi-Zariski exact sequence
\[ \HH_3(B\otimes_AB,l,l)\xrightarrow{\alpha_3} \HH_3(B,l,l) \to \HH_2(B\otimes_AB,B,l) \to \HH_2(B\otimes_AB,l,l)\xrightarrow{\alpha_2} \HH_2(B,l,l) \]
we have that $\alpha_3$ is surjective and $\alpha_2$ is injective. We deduce
\[\HH_2(B\otimes_AB,B,l)=0. \]
Then, from the Jacobi-Zariski exact sequence associated to $B\to B\otimes_AB \to B$ we obtain
\[\HH_1(B,B\otimes_AB,l)=0, \]
and since $f$ is flat,
\[ \HH_1(A,B,l)=\HH_1(B,B\otimes_AB,l)=0.\]
\end{proof}

\hfill \break
\subsection{A result of Radu, Andr\'e and Dumitrescu}
\hfill \break

Let $B$ be a noetherian local ring containing a field of characteristic $p>0$. A well known theorem of Kunz \cite{Kunz} tell us that $B$ is regular if the Frobenius homomorphism $\phi:B\to B$ is flat, and Rodicio \cite{Ro1} shows that it suffices to check that $\phi$ is of finite flat dimension. We are concerned here with the more general relative case (in order to see how the relative case implies the absolute one see Corollary \ref{absolute.Kunz}):

\begin{ques}
Let $f:A \to B$ be a flat homomorphism of local rings containing a field of characteristic $p>0$. Let $\Phi:{^\phi A}\otimes_AB\to {^\phi B}$, $\Phi(a\otimes b)=f(a)b^p$ be the relative Frobenius homomorphism (notation as in Proposition \ref{precision}). If $\mathrm{fd}_{{^\phi A}\otimes_AB}({^\phi B})<\infty$, is then $f$ regular?
\end{ques}

When $A$ and $B$ are noetherian the answer is in the affirmative. It was proved by N. Radu, M. Andr\'e and T. Dumitrescu (\cite{Rad}, \cite{An1993}, \cite{An1994} and \cite{Du}). We will prove the result in full generality:

\begin{thm}\label{relativekunz}
Let $f:A \to B$ be a flat homomorphism of rings containing a field of characteristic $p>0$. If $\mathrm{fd}_{{^\phi A}\otimes_AB}({^\phi B})<\infty$, then $f$ is discretely regular (and then regular).
\end{thm}
\begin{proof}
Let $\qqq$ be a prime ideal of $B$ and $\ppp=\Phi^{-1}({^\phi \qqq})$, so that $k(\qqq)$ is the residue field of $B$ at $\qqq$ and ${^\phi k(\qqq)}$ the residue field of ${^\phi B}$ at ${^\phi \qqq}$. By Theorem \ref{main}, the homomorphism
\[ \HH_2({^\phi A}\otimes_AB,{^\phi k(\qqq)},{^\phi k(\qqq)})=\HH_2(({^\phi A}\otimes_AB)_\ppp,{^\phi k(\qqq)},{^\phi k(\qqq)})\to \HH_2({^\phi B}_\qqq,{^\phi k(\qqq)},{^\phi k(\qqq)}) \]
is injective. Therefore the homomorphism in the Jacobi-Zariski exact sequence
\[\HH_2({^\phi A}\otimes_AB,{^\phi B},{^\phi k(\qqq)}) \to \HH_2({^\phi A}\otimes_AB,{^\phi k(\qqq)},{^\phi k(\qqq)})  \]
is zero. Then, from the commutative triangle
\[
\begin{tikzcd}
\HH_2({^\phi A}\otimes_AB,{^\phi B},{^\phi k(\qqq)})  \arrow[r, "0"] \arrow[dr, "\alpha"]  & \HH_2({^\phi A}\otimes_AB,{^\phi k(\qqq)},{^\phi k(\qqq)}) \arrow[d]\\
& \HH_1({^\phi A},{^\phi A}\otimes_AB,{^\phi k(\qqq)})
\end{tikzcd}
\]
we deduce $\alpha=0$, and so from the Jacobi-Zariski exact sequence we deduce that the homomorphism
\[  \HH_1({^\phi A},{^\phi A}\otimes_AB,{^\phi k(\qqq)}) \to \HH_1({^\phi A},{^\phi B},{^\phi k(\qqq)})  \]
is injective. Since $A\to B$ is flat, this homomorphism can be identified to the homomorphism induced by the absolute Frobenius homomorphisms of $A$ and $B$
\[ \HH_1(A,B,{^\phi k(\qqq)}) \to \HH_1({^\phi A},{^\phi B},{^\phi k(\qqq)})  \]
which is the zero homomorphism (\cite[Lemme 53]{An1988}). That is,
\[ \HH_1(A,B,k(\qqq))\otimes_{k(\qqq)}{^\phi k(\qqq)}=\HH_1(A,B,{^\phi k(\qqq)})=0 \]
and then $f$ is discretely regular.
\end{proof}

\begin{rem}
Note that the same proof works if we replace the hypothesis ``$f$ flat'' by ``$\phi_A:A\to A$ flat'', $\phi_A$ being the Frobenius homomorphism of $A$ (if $A$ is noetherian, that is equivalent to the regularity, at every maximal ideal, of $A$).
\end{rem}

Finally, we note that we can extend the original (absolute) Kunz's theorem to the non-noetherian case, since it is the particular case of Theorem \ref{relativekunz} when $A$ is the prime field of $B$:

\begin{cor}\label{absolute.Kunz}
Let $\localb$ be a local ring containing a field of characteristic $p>0$. If $\mathrm{fd}_B({^\phi B})<\infty$ then $B$ is formally regular for the discrete topology.
\end{cor}
\begin{proof}
Taking as $A$ the prime field of $B$, Theorem \ref{relativekunz} says that the Frobenius homomorphism $A\to B$ is discretely regular, and then Proposition \ref{prop.reg.hom} $(ii) \Rightarrow (iii)$ gives the result. Alternatively, we can apply Theorem \ref{main} to $\phi:B\to B$ noting that $\phi$ induces the zero map on $\HH_2(B,l,l)$.
\end{proof}

\end{document}